\documentclass[11pt,reqno]{article}
\usepackage{xifthen,setspace}
\usepackage{subfig}
\usepackage{amsmath,amsthm,amssymb,color}
\usepackage[numbers]{natbib}
\usepackage{booktabs}

\usepackage{pkgfile}

\usepackage{bbm}
\usepackage{mathrsfs}
\usepackage[breaklinks=true]{hyperref}
\usepackage{tikz}
\usetikzlibrary{arrows, automata}
\usetikzlibrary{calc}
\usetikzlibrary{positioning}

\numberwithin{equation}{section}
\allowdisplaybreaks[4]

\theoremstyle{plain}
\newtheorem{theorem}{Theorem}[section]

\newtheorem{lemma}{Lemma}[section]
\newtheorem{corollary}{Corollary}[section]

\theoremstyle{definition}
\newtheorem{definition}{Definition}[section]
\newtheorem{example}{Example}[section]
\newtheorem{remark}{Remark}[section]

\makeatletter

\newcount\minute
\newcount\hour
\newcount\hourMins
\def\now{%
\minute=\time%
\hour=\time \divide \hour by 60%
\hourMins=\hour \multiply\hourMins by 60%
\advance\minute by -\hourMins%
\zeroPadTwo{\the\hour}:\zeroPadTwo{\the\minute}%
}
\def\zeroPadTwo#1{\ifnum #1<10 0\fi#1}

\def\^#1{\ifmmode {\mathaccent"705E #1} \else {\accent94 #1} \fi}
\def\~#1{\ifmmode {\mathaccent"707E #1} \else {\accent"7E #1} \fi}

\def\*#1{#1^\ast}
\edef\-#1{\noexpand\ifmmode {\noexpand\bar{#1}} \noexpand\else \-#1\noexpand\fi}
\def\>#1{\vec{#1}}
\def\.#1{\dot{#1}}

\def\atop{\@@atop}
\def\*#1{\mathscr{#1}}

\renewcommand{\leq}{\leqslant}
\renewcommand{\geq}{\geqslant}
\newcommand{\eq}{\eqref}

\newcommand{\Var}{\mathop{\mathrm{Var}}\nolimits}
\newcommand{\Cov}{\mathop{\mathrm{Cov}}}

\def\ben#1{\begin{equation}#1\end{equation}}

\def\bm#1{\begin{multline*}#1\end{multline*}}

\def\beqn#1\eeqn{\begin{align}#1\end{align}}
\def\beq#1\eeq{\begin{align*}#1\end{align*}}

\usepackage{graphicx}
\usepackage{latexsym}
\usepackage{amsmath,amsthm,amssymb,amscd}
\usepackage{epsf,amsmath}

\renewcommand\section{\@startsection {section}{1}{\z@}%
{-3.5ex \@plus -1ex \@minus -.2ex}%
{1.3ex \@plus.2ex}%
{\center\small\sc\mathversion{bold}\MakeUppercase}}

\def\subsection#1{\@startsection {subsection}{2}{0pt}%
{-3.5ex \@plus -1ex \@minus -.2ex}%
{1ex \@plus.2ex}%
{\bf\mathversion{bold}}{#1}}

\def\subsubsection#1{\@startsection{subsubsection}{3}{0pt}%
{\medskipamount}%
{-10pt}%
{\normalsize\itshape}{\kern-2.2ex. #1.}}

\usepackage{bm}

\def\blfootnote{\xdef\@thefnmark{}\@footnotetext}

\makeatother

\title{\textbf{Normal Approximation and Fourth Moment Theorems for Monochromatic Triangles}}

\author{Bhaswar B. Bhattacharya\thanks{Department of Statistics, University of Pennsylvania, USA, \texttt{bhaswar@wharton.upenn.edu}} \and Xiao Fang\thanks{Department of Statistics, The Chinese University of Hong Kong, Hong Kong, \texttt{xfang@sta.cuhk.edu.hk}} \and Han Yan\thanks{Department of Statistics, The Chinese University of Hong Kong, Hong Kong, \texttt{1155105775@link.cuhk.edu.hk}}}
\date{}

\begin{document}

\maketitle

\small 
\noindent\textbf{Abstract}:  Given a graph sequence $\{G_n\}_{n \geq 1}$ denote by $T_3(G_n)$ the number of monochromatic triangles in a uniformly random coloring of the vertices of $G_n$ with $c \geq 2$ colors. This arises as a generalization of the birthday paradox, where $G_n$ corresponds to a friendship network  and $T_3(G_n)$ counts the number of triples of friends with matching birthdays. In this paper we prove a central limit theorem (CLT) for $T_3(G_n)$ with explicit error rates. The proof involves constructing a martingale difference sequence by carefully ordering the vertices of $G_n$, based on a certain combinatorial score function, and using a quantitive version of the martingale CLT. We then relate this error term to the well-known fourth moment phenomenon, which, interestingly, holds only when the number of colors $c \geq 5$. We also show that the convergence of the fourth moment is necessary to obtain a Gaussian limit for any $c \geq 2$, which, together with the above result, implies that the fourth-moment condition characterizes the limiting normal distribution of $T_3(G_n)$, whenever $c \geq 5$. Finally, to illustrate the promise of our approach, we include an alternative proof of the CLT for the number of monochromatic edges, which provides quantitative rates for the results obtained in \cite{BDM}. \\


\normalsize

\section{Introduction and Main Results}\label{sec1}

Let $G_n=(V(G_n), E(G_n))$ be a deterministic sequence of simple graphs with vertex set $V(G_n) = \{1, 2, \dots, |V(G_n)|\}$ and edge set $E(G_n)$. Denote by 
$A(G_n)=(a_{ij}(G_n))_{1\leq i,j\leq |V(G_n)| }$,  the adjacency matrix of $G_n$, namely, $a_{ij}(G_n)=1$ if $(i, j)$ is an edge in $G_n$ and $a_{ij}(G_n)=0$ otherwise. In a {\it uniformly random $c$-coloring of $G_n$}, the vertices of $G_n$ are colored with $c \geq 2$ colors as follows:  
\begin{equation}\P(v\in V(G_n) \text{ is colored with color } a\in \{1, 2, \ldots, c\})=\frac{1}{c},
\label{eq:uniform}
\end{equation}
independent from the other vertices. An edge $(a, b) \in E(G_n)$ is said to be {\it monochromatic} if $X_a=X_b$, where $X_v$ denotes the color of the vertex $v \in V(G_n)$ in a uniformly random $c$-coloring of $G_n$. Denote by  
\begin{align}\label{eq:me}
T_2(G_n)=\sum_{1\le s_1 < s_2 \leq |V(G_n)|} a_{s_1 s_2}(G_n) \bm 1\{X_{s_1}=X_{s_2}\},
\end{align} 
the number of monochromatic edges in $G_n$. 

The statistic $T_2(G_n)$ arises in various contexts, for example, as the Hamiltonian of the Ising/Potts model on $G_n$ \cite{bmpotts}, in non-parametric two-sample tests \cite{fr}, and as a generalization of the birthday paradox \cite{arriatia_garibaldi_kilian,BaHoJa92,CeFo06,ChDiMe05,dasguptasurvey,diaconisholmes,diaconismosteller}: If $G_n$ is a friendship-network graph colored uniformly with $c=365$ colors (corresponding to birthdays and  assuming the birthdays are uniformly distributed across the year), then two friends will have the same birthday whenever the corresponding edge in the graph $G_n$ is monochromatic. (When the underlying graph $G_n=K_n$ is the complete graph $K_n$ on $n$ vertices, this reduces to the classical birthday problem.) In particular, $T_2(G_n)$ counts the number of pairs of friends with matching birthdays. The birthday problem arise naturally many applications, for example, in the study of coincidences \cite{diaconismosteller}, testing discrete distributions \cite{batu,diakonikolas_collision}, and the discrete logarithm problem \cite{BBBDL,ghdl,pollarddlp}, all of which require  understanding the asymptotic properties of $T_2(G_n)$ for various graph sequences $G_n$. 

The limiting distribution of $T_2(G_n)$ exhibit various universality phenomena, depending on how the number of colors $c$ scales with the number of edges $|E(G_n)|$. In particular, Bhattacharya et al. \cite{BDM} showed that  for any $c \geq 2$ fixed, the asymptotic normality of the standardized version of $T_2(G_n)$ exhibits a {\it fourth-moment phenomenon}. More precisely, for any $c \geq 2$ and a sequence of graphs with $|E(G_n)| \rightarrow \infty$,
\begin{align}\label{eq:Z2Gn}
Z_2(G_n)=\frac{T_2(G_n)- \E(T_2(G_n))}{\sqrt{\Var(T_2(G_n))}} \stackrel{D} \rightarrow N(0, 1),
\end{align}
whenever $\E(Z_2(G_n)^4) \rightarrow 3$ (cf. \cite[Theorem 1.3]{BDM} and Section \ref{sec:Z2} below for further details).  This is an example of the celebrated fourth-moment phenomenon, which was originally discovered in the seminal papers \cite{NoPe09,NuPe05} and has, since then, emerged as a unifying principle in various problems,  
asserting that a central limit theorem (CLT) for non-linear functionals of random fields is often implied by the convergence of the corresponding sequence of fourth moments.  

A natural generalization of the birthday problem is to consider birthday matches between 3 or more friends \cite{extension}. This can be formulated as a graph coloring problem, where instead of counting monochromatic edges, one counts the number of monochromatic $r$-cliques (the complete graph on $r$-vertices) in a uniformly  random $c$-coloring of a friendship network $G_n$, for some $r \geq 3$. Hereafter, we denote by $T_r(G_n)$ the number of monochromatic $r$-cliques in a uniformly random $c$-coloring of $G_n$. In addition to its natural application in understanding coincidences \cite[Problem 3]{diaconismosteller}, this and related statistics arise in various problems in 
occupancy urns and cryptology (cf. \cite{BBBDL,ns} and the references therein).  Given the results for $T_2(G_n)$, it is natural to conjecture a similar fourth-moment phenomenon for the asymptotic normality of $T_r(G_n)$. However, the combinatorial techniques developed in \cite{BDM} for analyzing $T_2(G_n)$ does not generalize to monochromatic triangles or higher cliques, and standard Stein's method techniques do not appear to be robust enough for obtaining the precise conditions required, and, as a consequence,  proving such limit theorems have turned out to be surprisingly difficult. 

In this paper, we take the first step in this direction by considering the case $r=3$, which corresponds to counting monochromatic triangles. Our main result is a CLT for 
\begin{align}\label{eq:Z}
Z_3(G_n)=\frac{T_3(G_n)-\E(T_3(G_n))}{\sqrt{\Var(T_3(G_n))}}, 
\end{align} 
with an error bound for any $c \geq 2$ fixed (Theorem \ref{T1}). The proof uses the Hoeffding's decomposition to write $Z_3(G_n)$ as a martingale difference sequence, followed by an application of a quantitative version of the martingale CLT \cite{ HeBr70}. The highlight of the proof is the construction of the martingale difference sequence, which requires a careful ordering of the vertices of $G_n$, based on the counts of certain subgraphs passing through each of the  vertices, to control various crucial error terms. We then show how this error bound relates to the fourth moment condition. Here an interesting threshold behavior emerges: For any $c \geq 5$ fixed, the error bound obtained in Theorem \ref{T1} can be rewritten in terms of the fourth moment difference $\E(Z_3(G_n)^4) -3$, which shows, for $c \geq 5$, $Z_3(G_n) \stackrel{D} \rightarrow N(0, 1)$ whenever $\E(Z_3(G_n)^4) \rightarrow 3$ (Theorem \ref{T2}). However, this is not the case for $2 \leq c \leq 4$, where there are instances with $\E(Z_3(G_n)^4) \rightarrow 3$, but $Z_3(G_n)$ has a non-Gaussian limit. In order to obtain the critical value of $c$ beyond which the fourth-moment phenomenon holds,  one requires a delicate understanding of the fourth moment of $Z_3(G_n)$. While this involves a tedious calculation, it illustrates the intricacies of the problem as  one moves from monochromatic edges to triangles and beyond.  We also show that the fourth-moment condition is necessary to have a  normal limit for any $c \geq 2$ fixed (Theorem \ref{THM:4MOMENT}), which combined with the result above shows that the fourth-moment condition is necessary and sufficient for $Z_3(G_n)$ to have a Gaussian limit, whenever $c \geq 5$.  In fact, to the best of our knowledge, this is the first example of the fourth moment phenomenon that is not a degenerate $U$-statistic of a fixed order (see Remark \ref{remark:4moment} for further details). Finally, to illustrate the core idea behind of our approach, we apply this method to obtain quantitative bounds for the CLT of the number of monochromatic edges $Z_2(G_n)$ (Theorem \ref{T3}). As before, the proof requires a careful reordering of the vertices (in this case, based on the non-increasing order of the degrees), which strengthens, by providing explicit error rates, the results obtained in \cite{BDM}. The formal statements of the results are given below.

\subsection{Monochromatic Triangles: CLT and the Fourth Moment Phenomenon} 
\label{sec:Z3}

Recall that $T_3(G_n)$ denotes the number of monochromatic triangles in a uniformly random $c$-coloring of $G_n$. More formally, 
\begin{align}\label{eq:T3Gn}
T_3(G_n):= \sum_{1 \leq s_1 < s_2 < s_3 \leq |V(G_n)| } a_{s_1 s_2}(G_n) a_{s_2 s_3}(G_n) a_{s_1 s_3}(G_n) \bm 1\{ X_{s_1}= X_{s_2}= X_{s_3}\},
\end{align} 
where $X_v$ denotes the color of the vertex $v \in V(G_n)$ obtained from \eqref{eq:uniform}. For a fixed simple graph $H$, denote by $N(H, G_n)$ the number of copies (not necessarily induced) of the graph $H$ in $G_n$. Then it is easy to see that $\E(T_3(G_n))=\frac{N(K_3, G_n)}{c^2} $, where $K_3$ denotes a triangle. Hence, without loss of generality, hereafter, we assume $N(K_3, G_n) \geq 1$, for all $n \geq 1$. Moreover, a direct calculation shows that  
\begin{align}\label{eq:variance_triangle}
\Var(T_3(G_n))=\frac{1}{c^2} \left(1-\frac{1}{c^2}\right) N(K_3, G_n)+2\left(\frac{1}{c^3}-\frac{1}{c^4}\right) N(\triangle_2, G_n),
\end{align} 
where $\triangle_s$ denotes the graph formed by $s$-tuples of triangles sharing one common edge, for $s \geq 1$. We will refer to the graph $\triangle_s$ as the $s$-pyramid (see Figure \ref{fig2} in Appendix \ref{app2} for illustrations of $\triangle_1$, $\triangle_2$, $\triangle_3$, and $\triangle_4$). Note that $\triangle_1$ is isomorphic to the triangle $K_3$, hence, both these notations will be used interchangeably. 

Our main result is a CLT along with an error bound for the standardized version of $T_3(G_n)$ (recall the definition of $Z_3(G_n)$ from \eqref{eq:Z}). To this end, denote 
\begin{align*}
b(G_n):=\sum_{1\leq s_1< s_2< s_3< s_4 \leq |V(G_n)|} \left( d_{s_1 s_2}  d_{s_2 s_3}  d_{s_3 s_4}  d_{s_4 s_1} + d_{s_1 s_2}  d_{s_2 s_4}  d_{s_4 s_3}  d_{s_3 s_1} +d_{s_1 s_3}  d_{s_3 s_2}  d_{s_2 s_4}  d_{s_4 s_1}   \right),
\end{align*}
where $d_{s_1 s_2} := \sum_{s_3 \ne \{s_1, s_2\}} a_{s_1 s_2}(G_n) a_{s_2 s_3}(G_n) a_{s_3 s_1}(G_n)$ is the number of triangles in $G_n$ with $(s_1, s_2)$ as an edge. Also, denote by $\Phi$ the standard normal distribution function. 

\begin{theorem}\label{T1} For $Z_3(G_n)$ as defined in \eqref{eq:Z}, 
\begin{align}\label{12}
\sup_{x\in \mathbb{R}} |\P(Z_3(G_n)\leq x)-\Phi(x)| \leq & K \left[  R_1^{\frac{1}{4}}  +  R_2 \right]^{\frac{1}{5}},
\end{align}
where 
\begin{align}\label{eq:R12}
R_1:= \frac{1+N(\triangle_4, G_n)}{( N(K_3, G_n) + N(\triangle_2, G_n))^2}, \quad R_2:= \frac{b(G_n)}{(  N(K_3, G_n) + N(\triangle_2, G_n))^2}, 
\end{align} 
and $K := K(c) >0$ is a constant that only depends on the number of colors $c$. 
\end{theorem}

Note that Theorem \ref{T1} shows that for any $c \geq 2$ fixed, $Z_3(G_n) \stackrel{D} \rightarrow  N(0, 1)$, whenever 
$$N(\triangle_4, G_n)=o((N(K_3, G_n)+N(\triangle_2, G_n))^2) \text{ and } b(G_n)=o((N(K_3, G_n)+N(\triangle_2, G_n))^2).$$ These conditions are, in fact, tight in the sense that, if $N(\triangle_4, G_n)$ or $b(G_n)$ is \emph{not} of smaller order than $(N(K_3, G_n)+N(\triangle_2, G_n))^2$, then the CLT \emph{may not} hold for $Z_3(G_n)$ (as shown in Examples \ref{example1} and \ref{example2}).  The proof of Theorem \ref{T1}, which is given in Section \ref{sec3}, proceeds by writing $Z_3(G_n)$ as a martingale difference sequence (using the Hoeffding's decomposition), and then applying the martingale CLT with error bounds from \cite{HeBr70} to $Z_3(G_n)$.
The resulting error bound involves the two terms: one involving $b(G_n)$ and the other involving 
\begin{align}\label{eq:sGn}
s(G_n)=\sum_{1 \leq s_1 < s_2 < s_3 \leq |V(G_n)|} d_{s_1 s_3}^2 d_{s_2 s_3}^2. 
\end{align}
One of the main  combinatorial ingredients of the proof is to show that, after a careful ordering of the vertices of the graph, $s(G_n)$ can be bounded in terms of the pyramids $N(\triangle_s, G_n)$, for $1 \leq s \leq 4$ (see Lemma \ref{l1}). In fact, both the quantities $s(G_n)$ and $b(G_n)$ can be interpreted as the counts of certain subgraphs in $G_n$ which arise in the fourth moment of $Z_3(G_n)$, as will be evident from the proof of Theorem \ref{T1}. 

Next, we discuss the connection of the above result with the fourth moment phenomenon.  Recall that the asymptotic normality of the number of monochromatic edges exhibits the fourth moment phenomenon, that is, $Z_2(G_n)$ converges to $N(0, 1)$ whenever $\E (Z_2(G_n)^4 ) \rightarrow 3$, for all $c \geq 2$ fixed (see Remark \ref{remark:Z2} below for a more detailed discussion of this result). Therefore, it is natural to wonder whether this phenomenon extends to monochromatic triangles. This is discussed in the following theorem.\footnote{For positive sequences $\{a_n\}_{n\geq 1}$ and $\{b_n\}_{n\geq 1}$, $a_n \lesssim b_n$ means $a_n \leq C_1 b_n$, and $a_n \gtrsim b_n$ means $a_n \geq C_2 b_n$, and $a_n \asymp b_n$ means $C_2 b_n \leq a_n \leq C_1 b_n$, for all $n$ large enough and positive constants $C_1, C_2$. Moreover, subscripts in the above notation,  for example $\lesssim_\square$ and $\gtrsim_\square$,  denote that the hidden constants may depend on the subscripted parameters. }  

\begin{theorem}\label{T2}
For $Z_3(G_n)$ as in \eq{eq:Z} the following hold: 

\begin{itemize}

\item[$(1)$] For any $c\geq 5$ fixed, 
\ben{\label{13}
\sup_{x\in \mathbb{R}} |\P(Z_3(G_n)\leq x)-\Phi(x)| \lesssim_c (\E (Z_3(G_n)^4)-3)^{\frac{1}{20}}. } 

\item[$(2)$] If $2 \leq c \leq 4$, there exists a sequence of graphs $\{G_n\}_{n \geq 1}$, with $N(K_3, G_n) \rightarrow \infty$, for which $\E (Z_3(G_n)^4) \rightarrow 3$, but $Z_3(G_n)$ does not converge in distribution to $N(0, 1)$. 

\end{itemize}  
\end{theorem}

The result above shows that, unlike for monochromatic edges where the fourth moment phenomenon holds for all $c \geq 2$, the fourth moment phenomenon for monochromatic triangles is more subtle. Here, $\E (Z_3(G_n)^4) \rightarrow 3$ implies the asymptotic normality of $Z_3(G_n)$ only if $c \geq 5$, in which case we can get a quantitative error rate as in \eqref{13}. To understand why the fourth moment phenomenon fails for $2 \leq c \leq 4$,  we compute in Lemma \ref{l2} the fourth moment difference $\E(Z_3(G_n)^4) - 3$ precisely. This involves keeping track of the different subgraphs (and their coefficients) which arise when the fourth moment is expanded out as a sum (over the various graphs formed by the union of 4 triangles). The calculations are tedious, but in the end we arrive at the following rather surprising observation: For $c\geq 5$, all the subgraph coefficients in the fourth moment difference are positive, however, for $2 \leq c \leq 4$ certain coefficients can be negative. Therefore, for $2 \leq c \leq 4$, it is possible to construct graphs such that these coefficients cancel each other and the fourth moment converges to 3, but the corresponding subgraph counts are too large for a CLT to hold (see Section \ref{sec4} details). It remains open to show whether counts of general monochromatic cliques or subgraphs have a CLT and fourth moment phenomenon as in \eqref{13} (see the discussion in Section \ref{sec6} for more details on this problem and future directions).

\begin{remark}\label{remark:4moment} 
The fourth moment phenomenon was first discovered by Nualart and Peccati \cite{NuPe05}, who showed that the convergence of the first, second, and fourth moments to $0, 1$, and $3$, respectively, guarantees asymptotic normality for a sequence of multiple stochastic Wiener-It\^o integrals of fixed order. Later, Nourdin and Peccati \cite{NoPe09} provided an error bound for the fourth moment theorem of \cite{NuPe05}. This prompted a  wave of major developments and, over the years, the fourth moment phenomenon has emerged as a  ubiquitous principle governing the central limit theorems for various non-linear functionals of random fields.  We refer the reader to the book \cite{book_np} for an introduction to the topic and  website \url{https://sites.google.com/site/malliavinstein/home} for a list of the recent results.  Incidentally, related results for degenerate $U$-statistics of a fixed order of independent random variables were first obtained by de Jong \cite{de87, de90}. Here, in addition to the fourth moment condition, in general, an extra condition is  needed to control the maximum influence of the underlying independent random variables (cf.  \cite[Theorem 2.1]{de87} and \cite[Theorem 1.6]{DoKr19}). However, to the best of our knowledge, Theorem~\ref{T2} is the first example of the fourth moment phenomenon that is not a degenerate $U$-statistic of a fixed order. In addition, we do not need the extra condition of \cite{de87, de90} and \cite{DoPe17} controlling the maximum influence of the underlying independent random variables (see the discussion following Theorem \ref{T3} for more details on this condition). 
\end{remark}

Another natural question is whether the convergence of the fourth moment necessary for the CLT of $Z_3(G_n)$. We answer this question in the affirmative in the following theorem

\begin{theorem}\label{THM:4MOMENT}
Let $Z_3(G_n)$ be as defined in \eq{eq:Z}. Then for $c\geq 2$ fixed, $Z_3(G_n) \stackrel{D} \rightarrow N(0, 1)$ implies $\E (Z_3(G_n)^4) \to 3$.
\end{theorem}

This result shows the necessity of the fourth moment condition, for all  $c \geq 2$ fixed. This combined with Theorem \ref{T2} above, shows that,   for any fixed $c \geq 5$, $Z_3(G_n) \stackrel{D} \rightarrow N(0, 1)$ if and if $\E (Z_3(G_n)^4) \to 3$.  The proof of Theorem \ref{THM:4MOMENT} is given in Section \ref{sec:4momentpf}. The proof shows that, for any $c \geq 2$, all the moments of $Z_3(G_n)$ are bounded, which implies the convergence of the fourth moments by uniform integrability. To show this, we use an estimate from extremal combinatorics which bounds the number of copies of a hypergraph $F$ in another weighted hypergraph $H$ in terms of the fractional stable number of $F$ (see Corollary \ref{cor:alon_exponent} in Appendix \ref{sec:moment_hypergraph} for the precise statement).

\subsection{Quantitative Bounds for Monochromatic Edges} 
\label{sec:Z2}

As mentioned before, the precise conditions for the asymptotic normality of the number of monochromatic edges are well-understood \cite{BDM}. For instance, when $c$ is fixed, then  \cite[Theorem 1.3]{BDM} shows that $Z_2(G_n) \stackrel{D} \rightarrow N(0, 1)$ if and only if $N(C_4, G_n)=o(|E(G_n)|^2)$, where $C_s$ denotes the cycle of length $s \geq 3$, which is, in fact, equivalent to the fourth moment condition $\E(Z_2(G_n)^4) \rightarrow 3$. However, the proofs in \cite{BDM} used the method of moments and do not provide any rate of convergence. To showcase the promise our approach and quantify the asymptotic results in \cite{BDM}, we apply the martingale CLT approach described above to $Z_2(G_n)$. 

\begin{theorem}\label{T3} Let $Z_2(G_n)$ be as defined in \eqref{eq:Z2Gn}. Then, 
\begin{align}\label{eq:Z2CLT}
\sup_{x \in \R} |\P(Z_2(G_n) \leq x)-\Phi(x)|\leq K \left(\frac{c}{|E(G_n)|} + \frac{1}{\sqrt{|E(G_n)|}}  + \frac{N(C_4, G_n)}{c |E(G_n)|^2} \right)^{\frac{1}{5}},
\end{align} 
where $K >0$ is a universal constant (not depending on $n$ and $c$). 
\end{theorem}

The proof of Theorem \ref{T3} is given below in Section \ref{sec2}. In this case, $Z_2(G_n)$ is a degenerate $U$-statistic of order 2, and the proof proceeds by writing $Z_2(G_n)$ as a martingale difference sequence and using a quantitive version of the martingale CLT \cite{HeBr70}, as before. The martingale CLT was used by de Jong \cite{de87, de90} to deal with degenerate $U$-statistics of a fixed order of independent random variables. More recently, D\"obler and Peccati \cite{DoPe17} used Stein's method to prove an error bound for de Jong's CLT. However, this error bound is not directly applicable in our problem, because it involves an additional term controlling the maximum influence (denoted by $\rho_n$ in  \cite[Theorem 1.3]{DoPe17}) which, in this case is proportional to $\Delta(G_n)/|E(G_n)|$, where $\Delta(G_n)$ is the maximum degree of $G_n$. It is easy to see that it is not necessary for this term to vanish for the CLT of $Z_2(G_n)$ to hold: For example, if $G_n=K_{1, n}$ is the star graph on $n$ vertices, with the central vertex labelled $1$ and the other vertices  labeled $\{2, 3, \ldots, n\}$, then $\Delta(G_n)/|E(G_n)|=O(1)$. Nevertheless, the CLT holds, because in this case $Z_2(K_{1, n})$ is a sum of independent random variables (note that the collection of edge indicators $\{\bm 1\{X_1= X_v\}\}_{1 \leq v \leq n-1}$ are independent).  Hence, the bound in \cite{DoPe17} is not directly applicable to our problem. We circumvent this issue by carefully ordering the vertices of $G_n$ while constructing the martingale, which ensures the term involving the maximum degree does not arise in the martingale CLT error terms. 

Note that in  Theorem \ref{T3} the dependence on the number of colors $c$ has been made explicit. Therefore, this result holds for any $c \geq 2$, fixed or depending on $n$. We discuss the consequences in the two cases separately: 

\begin{itemize} 

\item {\it $c \geq 2$ is fixed}: Here, by absorbing the dependence on $c$ in to the leading constant, the error term in \eqref{eq:Z2CLT} can be simplified as 
\begin{align}\label{eq:Z2CLT_I}
\sup_{x \in \R} |\P(Z_2(G_n) \leq x)-\Phi(x)|\lesssim_c \left( \frac{1}{\sqrt{|E(G_n)|}}  + \frac{N(C_4, G_n)}{ |E(G_n)|^2} \right)^{\frac{1}{5}}.
\end{align} 
This shows that, when $c$ is fixed, $Z_2(G_n) \stackrel{D} \rightarrow  N(0, 1)$, whenever $|E(G_n)| \rightarrow \infty$ such that $N(C_4, G_n)=o(|E(G_n)|^2)$. Thus, Theorem \ref{T3} not only recovers the result in \cite[Theorem 1.3]{BDM}, it provides an explicit rate of convergence. 

\item {\it $c=c_n \rightarrow \infty$ such that $|E(G_n)|/c \rightarrow \infty$}: In this case, using the  bound $N(C_4, G_n) \lesssim |E(G_n)|^2$ (see, for example, \cite[Theorem 1]{alon81}), the error term in \eqref{eq:Z2CLT} can be simplified as 
\begin{align}\label{eq:Z2CLT_II}
\sup_{x \in \R} |\P(Z_2(G_n) \leq x)-\Phi(x)|\lesssim \left( \frac{c}{|E(G_n)|} + \frac{1}{\sqrt{|E(G_n)|}}  + \frac{1}{c} \right)^{\frac{1}{5}}.
\end{align} 
Again, this reaffirms \cite[Theorem 1.2]{BDM} which shows that $Z_2(G_n) \stackrel{D} \rightarrow  N(0, 1)$, whenever  $c \rightarrow \infty$ such that $|E(G_n)|/c \rightarrow \infty$. The rate in \eqref{eq:Z2CLT_II} is, in general, worse than that in \cite[Theorem 1.1]{Fa15}, which was obtained by an application of Stein's method for normal approximation. However, the technique in \cite{Fa15} was unable to recover the precise conditions for asymptotic normality in the case where  $c$ is fixed. On the other hand, Theorem \ref{T3} provides a unified proof for the case $c$ is fixed and $c \rightarrow \infty$, and, as discussed above, obtains the exact conditions for the asymptotic normality of $Z_2(G_n)$ in both cases. 
\end{itemize}
Finally, because the conditions for the limiting normality of $Z_2(G_n)$ arising from \eqref{eq:Z2CLT_I} and \eqref{eq:Z2CLT_II} above, are equivalent to the fourth moment condition, the error term in \eqref{eq:Z2CLT} can be bounded in terms of fourth moment difference $\E(Z_2(G_n)^4) - 3$ (see Remark \ref{remark:Z2}).

\subsection{Organization} The rest of the paper is organized as follows: The proof of Theorem \ref{T3} is given in Section~\ref{sec2}. In Section \ref{sec3} we prove Theorem \ref{T1} and provide examples illustrating the necessity of the error terms in \eqref{12}. The proofs of Theorem \ref{T2} and Theorem \ref{THM:4MOMENT} are given in Section \ref{sec4} and Section \ref{sec:4momentpf}, respectively. In Section \ref{sec6} we summarize our results and discuss future directions. Few technical details are given in the Appendix.


\section{Proof of Theorem \ref{T3}}\label{sec2}

Assume without loss of generality that the vertices of $G_n$ are labelled $\{1, 2, \ldots, |V(G_n)|\}$ in non-increasing order of degrees, that is, $\mathrm{deg}(1) \geq \mathrm{deg}(2) \geq \ldots \geq \mathrm{deg}(|V(G_n)|)$, where $\mathrm{deg}(v)$ denotes the degree of the vertex $v \in \{1, 2, \ldots, |V(G_n)|\}$.  We now follow \cite{de87} and write $Z_2(G_n)$ (recall definition from \eqref{eq:Z2Gn}) as a sum of martingale differences. Throughout this proof we will denote 
\begin{align}\label{eq:varT2Gn}
\sigma^2=\Var(T_2(G_n))= \frac{|E(G_n)|}{c}\left(1-\frac{1}{c}\right), 
\end{align}
where the last equality above follows from the definition of $Z_2(G_n)$ and noting that the covariance of $\bm 1\{X_{s_1}=X_{s_2}\}-\frac{1}{c}$ and $\bm 1\{X_{s_1}=X_{s_3}\}-\frac{1}{c}$ is zero, for $1 \leq s_1 < s_2 < s_3 \leq |V(G_n)|$. Now, define $ W_{s t}=a_{st}(G_n)(\bm 1\{X_{s}=X_{t}\}-\frac{1}{c})$, for $1 \leq s < t \leq |V(G_n)|$, and write $Z_2(G_n)$ as 
$$Z_2(G_n)=\sum_{1\leq t \leq |V(G_n)| }U_{t} \quad \text{where} \quad U_t=\frac{1}{\sigma} \sum_{s: s < t} W_{st}.$$
Note that $\{U_t\}_{1 \leq t \leq |V(G_n)| }$ is a martingale difference sequence with respect to the filtration sequence $\mathcal{F}_t=\sigma(\{X_1,\dots, X_t\})$. To see this note that 
$$\E(U_t|\cF_{t-1})=\frac{1}{\sigma} \sum_{s: s < t} a_{st}(G_n)\E\left(\bm 1\{X_s=X_t\} -\frac{1}{c}\Big|X_s\right)= 0.$$
Now, from display (1) of \cite{HeBr70} (with $s_n=1$, $\delta=1$, $Y_t=U_t$ and $\sigma_t^2=E(U_t^2|X_1,\dots, X_{t-1})$, for $1 \leq t \leq |V(G_n)|$), we have
\begin{align}\label{1}
\sup_x |\P(Z_2(G_n) \leq x)-\Phi(x)| & =\sup_x \left|\P\left(\sum_{1\leq t\leq |V(G_n)|} U_t\leq x\right)-\Phi(x)\right| \nonumber \\ 
& \lesssim \left\{  \sum_{t=1}^{|V(G_n)|} \E(U_t^4)+ \Var\left(\sum_{t=1}^{|V(G_n)|} U_t^2\right)  \right\}^{\frac{1}{5}}.
\end{align}

In order to bound the two terms on the RHS above, we need the following estimates, the second of which is due to the crucial fact that the vertices are ordered in the non-increasing order of the degrees. 

\begin{lemma}\label{lm:2} For $\{ W_{s_1 s_2}\}_{1 \leq s_1 < s_2 \leq |V(G_n)| }$ as defined above, the following hold: 
\begin{enumerate}
\item[(a)] $\sum_{1\leq s_1< s_2 \leq |V(G_n)| }\E( W_{s_1 s_2}^4) \lesssim \frac{|E(G_n)|}{c}$. 

\item[(b)] $\sum_{1\leq s_1 < s_2 < s_3\leq |V(G_n)| } \E(W_{s_1 s_3}^2 W_{s_2 s_3}^2) \lesssim  \frac{|E(G_n)|^{\frac{3}{2}}}{c^2}$,

\item[(c)] $\sum_{1\leq s_1 < s_2 < s_3< s_4 \leq |V(G_n)| } \E(W_{s_1 s_3}W_{s_2 s_3}W_{s_1 s_4 }W_{s_2  s_4 }) \lesssim \frac{N(C_4, G_n)}{c^3}$. 
\end{enumerate}
\end{lemma}

\begin{proof} To begin with note that 
$$\sum_{1\leq s_1< s_2 \leq |V(G_n)| }\E( W_{s_1 s_2}^4 ) =  \sum_{1\leq s_1< s_2 \leq |V(G_n)| } a_{s_1 s_2}(G_n) \E\left(\bm 1\{X_{s_1}=X_{s_2} \} -\frac{1}{c} \right)^4 \lesssim \frac{|E(G_n)|}{c},$$
because the leading term in $\E(\bm 1\{X_{s_1}=X_{s_2}\} -\frac{1}{c})^4$ is $O(1/c)$, since $c \geq  2$. This proves (a). 

Similarly for part (c), $\sum_{1\leq s_1 < s_2 < s_3< s_4\leq |V(G_n)| } \E(W_{s_1 s_3}W_{s_2 s_3}W_{s_1 s_4 }W_{s_2 s_4 }) \lesssim \frac{N(C_4, G_n)}{c^3}$, since the leading term is given by $\E(\bm 1\{X_{s_1}=X_{s_2}=X_{s_3}=X_{s_4}\})=\frac{1}{c^3}$.  
 
For (b), using $\mathrm{deg}(s_2)\geq \mathrm{deg}(s_3)$, gives 
\begin{align}
\sum_{1\leq s_1 < s_2 < s_3\leq |V(G_n)| } \E(W_{s_1 s_3}^2 W_{s_2 s_3}^2) & \lesssim \frac{1}{c^2} \sum_{1\leq s_1 < s_2 < s_3\leq |V(G_n)| }a_{s_1 s_3}(G_n) a_{s_2 s_3}(G_n) \nonumber \\ 
& \leq \frac{1}{c^2} \sum_{1 < s_2<s_3\leq |V(G_n)| }  a_{s_2 s_3}(G_n) \mathrm{deg}(s_3) \nonumber \\
& \leq \frac{1}{c^2}\sum_{(u, v) \in E(G_n)} \mathrm{deg}(u) \wedge \mathrm{deg}(v) \nonumber \\ 
& \lesssim \frac{|E(G_n)|^{\frac{3}{2}}}{c^2} , \nonumber 
\end{align}
where last step uses the inequality $\sum_{(u, v) \in E(G_n)} \mathrm{deg}(u) \wedge \mathrm{deg}(v) \leq \sqrt{2} |E(G_n)|^{3/2}$, from \cite[Page 37]{BaHoJa92}. 
\end{proof}

We now proceed to bound the two terms on the RHS of \eqref{1}. We begin with the     first term on the RHS of \eq{1},  
\begin{align}
\sum_{1\leq t \leq |V(G_n)| } \E (U_{t}^4)=&\frac{1}{\sigma^4}\sum_{1\leq t\leq |V(G_n)| }\E\left(\sum_{s: s<t}W_{s t} \right)^4 \nonumber \\
=&\frac{1}{\sigma^4} \sum_{1\leq t \leq |V(G_n)| } \E\left(\sum_{s: s < t}W_{s t}^2+2\sum_{s, s': s < s' < t}W_{s t}W_{s' t}\right)^2 \nonumber \\
=&\frac{1}{\sigma^4} \sum_{1\leq t \leq |V(G_n)| } \left(\sum_{s: s < t} \E W_{s t}^4+6\sum_{s, s': s < s' < t} \E W_{s t}^2W_{s' t}^2 \right) \tag*{(since the cross product terms have expectation zero)} \nonumber \\
\lesssim & \frac{c^2}{|E(G_n)|^2} \left(\frac{|E(G_n)|}{c} + \frac{|E(G_n)|^{\frac{3}{2}}}{c^2} \right) \tag*{(using Lemma \ref{lm:2} (a) and (b),  and $\sigma\asymp \left(\frac{|E(G_n)|}{c}\right)^\frac{1}{2}$)} \nonumber \\ 
\label{eq:4}=& \frac{c}{|E(G_n)|} + \frac{1}{\sqrt{|E(G_n)|}}. 
\end{align} 

For the second term on the RHS of \eq{1},  we have 
\begin{align}\label{eq:5}
\Var\left(\sum_{t=1}^{|V(G_n)|} U_{t}^2\right) =& \frac{1}{\sigma^4} \Var\left(\sum_{1 \leq s_1< s_2 \leq |V(G_n)| } W_{s_1 s_2}^2+2\sum_{1 \leq s_1 < s_2 < s_3 \leq |V(G_n)| } W_{s_1 s_3}W_{s_2 s_3}\right)\\
\lesssim &   \frac{1}{\sigma^4} \left[\Var\left(\sum_{1 \leq s_1 <  \leq |V(G_n)| } W_{s_1 s_2}^2 \right)+\Var\left(\sum_{1 \leq s_1 < s_2 < s_3 \leq |V(G_n)| } W_{s_1 s_3}W_{s_2 s_3}\right) \right]. 
\end{align}
Now, by ruling out all the zero-covariance terms, we get 
\begin{align}\label{eq:6}
\frac{1}{\sigma^4}  \Var\left(\sum_{1 \leq s_1< s_2 \leq |V(G_n)| } W_{s_1 s_2}^2 \right)= \frac{1}{\sigma^4}  \sum_{1 \leq s_1 < s_2 \leq |V(G_n)| } \Var(W_{s_1 s_2}^2)  & \leq \frac{1}{\sigma^4}  \sum_{1 \leq s_1 < s_2 \leq |V(G_n)| } \E(W_{s_1 s_2}^4) \nonumber \\ 
& \lesssim \frac{c}{|E(G_n)|}, 
\end{align}
and
\begin{align}\label{eq:7}
\frac{1}{\sigma^4} & \Var\left(\sum_{1\leq s_1 < s_2 < s_3 \leq |V(G_n)| } W_{s_1 s_3}W_{s_2 s_3} \right)  \nonumber \\ 
& = \frac{1}{\sigma^4} \left(  \sum_{1 \leq s_1 < s_2 < s_3 \leq |V(G_n)| } \E(W_{s_1 s_3}^2W_{s_2 s_3}^2) + 2\sum_{1 \leq s_1 < s_2 < s_3< s_4 \leq |V(G_n)| } \E (W_{s_1 s_3}W_{s_2 s_3}W_{s_1 s_4 }W_{s_2 s_4} ) \right) \nonumber \\ 
& \lesssim  \frac{1}{\sqrt{|E(G_n)|}} + \frac{N(C_4, G_n)}{c |E(G_n)|^2},  
\end{align}
where the last step uses Lemma \ref{lm:2} (a) and (c),  and $\sigma\asymp \left(\frac{|E(G_n)|}{c}\right)^\frac{1}{2}$. 

Plugging in \eqref{eq:4}, \eqref{eq:5}, \eqref{eq:6}, and \eqref{eq:7} to the RHS of \eqref{1}, the result follows. \hfill $\Box$

\begin{remark}\label{remark:Z2} The error term in \eqref{eq:Z2CLT} can be expressed in terms of fourth-moment difference $\E(Z_2(G_n)^4) - 3$. To this end, recall, from \eqref{eq:varT2Gn}, that  $\sigma^2=\frac{|E(G_n)|}{c}(1-\frac{1}{c}) \asymp \frac{|E(G_n)|}{c}$. Then by a direct calculation,  
$$\E(Z_2(G_n)^4) - 3 = \frac{\gamma_1 |E(G_n)| +  \gamma_2 N(K_3, G_n) + \gamma_3 N(C_4, G_n)}{\sigma^4} ,$$ 
where $\gamma_1=\frac{1}{c}\left(1-\frac{7}{c}+\frac{12}{c^2}-\frac{6}{c^3} \right)$, $\gamma_2=\frac{36}{c^2} \left(1-\frac{1}{c}\right) \left(1-\frac{2}{c}\right)$, and $\gamma_3=\frac{24}{c^3}\left(1-\frac{1}{c}\right)$. 
Now, using the well-known bound $N(K_3, G_n)\lesssim |E(G_n)|^{3/2}$ (see display (1) of \cite{alon81}) gives, for any $c \geq 2$, 
$$\left|\E(Z_2(G_n)^4) - 3 - \frac{\gamma_3 N(C_4, G_n) }{\sigma^4}  \right| \lesssim \frac{c}{|E(G_n)|} +  \frac{1}{\sqrt{|E(G_n)|}}.$$
Then using $\frac{\gamma_3 N(C_4, G_n) }{\sigma^4} \asymp \frac{N(C_4, G_n)}{c|E(G_n)|^2}$ together   with \eq{eq:Z2CLT}, implies
\begin{align*}
\sup_{x \in \R} |\P(Z_2(G_n) \leq x)-\Phi(x)|\lesssim |\E(Z_2(G_n)^4) - 3|^{\frac{1}{5}}+\left(\frac{c}{|E(G_n)|}+\frac{1}{\sqrt{|E(G_n)|}}\right)^{\frac{1}{5}},
\end{align*}
hence, $\E (T_2(G_n))= \frac{|E(G_n)|}{c} \to \infty$ and $\E(Z_2(G_n)^4)\to 3$ imply $Z_2(G_n) \to N(0,1)$ in distribution.
\end{remark}

\section{Proof of  Theorem~\ref{T1}}\label{sec3}

We begin by recalling the Hoeffding's decomposition of a square integrable function of independent random variables. 

\begin{definition}(\cite{Ho48}) Suppose $W$ is a square integrable function of the independent random variables $\{X_1,\dots, X_{|V(G_n)|}\}$. Then the Hoeffding's decomposition of $W$ is 
$$W=\sum_{I \subset \{1,\dots,|V(G_n)|\}} W_I,$$ 
such that 
\begin{itemize}
\item[(a)]$W_I$ is $\mathcal{F}_I$-measurable, where $\mathcal{F}_I$ is the $\sigma$-field generated by $\{X_i: i\in I\}$, and

\item[(b)] $\E(W_I | \mathcal{F}_J)=0$ almost surely, unless $I\subset J$.
\end{itemize}  
In fact, $W_I$ is almost surely uniquely determined by the above conditions and is given by $$W_I=\sum_{J\subset I} (-1)^{|I|-|J|} E(W | \mathcal{F}_J).$$
\end{definition}

We now begin the proof of Theorem \ref{T1}. Throughout this proof, recalling \eqref{eq:variance_triangle}, we will denote 
\begin{align}\label{eq:variance_T3}
\sigma^2:=\Var(T_3(G_n))=\frac{1}{c^2} \left(1-\frac{1}{c^2}\right) N(K_3, G_n)+2\left(\frac{1}{c^3}-\frac{1}{c^4}\right) N(\triangle_2, G_n). 
\end{align}
It is straightforward to compute the Hoeffding decomposition of $Y:=Z_3(G_n)$ in \eq{eq:Z} to be 
\begin{align}\label{eq:Z3_H}
Y=\frac{1}{\sigma}\left(\sum_{1 \leq s_1 < s_2 \leq |V(G_n)|} Y_{s_1 s_2} +  \sum_{1 \leq s_1 < s_2 < s_3 \leq |V(G_n)|} Y_{s_1 s_2 s_3}\right),
\end{align} 
where $Y_{s_1 s_2}=d_{s_1 s_2} \left(\frac{1}{c} \bm 1\{X_{s_1}=X_{s_2}\}-\frac{1}{c^2} \right)$, (recall $d_{s_1 s_2}=\sum_{s_3 \notin \{s_1, s_2\}} a_{s_1 s_2 s_3}(G_n)$, is the number of triangles with $(s_1, s_2)$ as an edge), and 
$$Y_{s_1 s_2 s_3}= a_{s_1 s_2 s_3}(G_n) \left( \left\{ \bm 1\{X_{s_1}=X_{s_2}=X_{s_3} \}-\frac{1}{c^2} \right\} -\frac{1}{c} \sum_{1 \leq a < b \leq 3} \left\{\bm 1\{X_{s_a}=X_{s_b}\} -\frac{1}{c} \right\} \right),$$
with $a_{s_1 s_2 s_3}(G_n):= a_{s_1 s_2}(G_n) a_{s_2 s_3}(G_n) a_{s_3 s_1}(G_n)$. 

Now, let 
\begin{align}\label{eq:U}
U_t=\frac{1}{\sigma}\left (\sum_{s: s< t} Y_{s t} +  \sum_{s, s': s < s' < t} Y_{s s' t} \right).
\end{align} 
Then $Y=\sum_{t=1}^{|V(G_n)|} U_t$, and from property (b) of the Hoeffding decomposition, $\{U_t\}_{t \geq 1}$ is a martingale difference sequence. Therefore, from \eq{1}, it suffices to bound
\begin{align}\label{eq:AB}
A=\sum_{t=1}^{|V(G_n)|} \E (U_t^4) \quad \text{ and } \quad B= \Var\left(\sum_{t=1}^{|V(G_n)|} U_t^2 \right).
\end{align} 
Note that the random variables $U_1,\dots, U_{|V(G_n)|}$ depend on the ordering of the vertices of $G_n$, which is arbitrary. A crucial ingredient in the proof is the following combinatorial lemma, which shows that there is a particular ordering of the vertices of $G_n$ which ensures $s(G_n)$ (recall definition in \eqref{eq:sGn}), a quantity which arises when expanding the terms $A$ and $B$ in \eqref{eq:AB}, can be bounded in terms of the counts of the pyramids $N(\triangle_s, G_n)$, for $1 \leq s \leq 4$. 
 
\begin{lemma}\label{l1} Let $s(G_n)$ be as defined in \eqref{eq:sGn}. Then there exists an ordering of the vertices $\{1,\dots, |V(G_n)|\}$ such that 
$$s(G_n) \lesssim (N(\triangle_1, G_n)+N(\triangle_2, G_n))^{\frac{3}{2}}(1+N(\triangle_4, G_n))^{\frac{1}{4}}.$$
\end{lemma}

\begin{proof} For  $1\leq s_1 < s_2 < s_3\leq |V(G_n)|$ fixed we  introduce the following symbols: 
\begin{itemize}

\item $\triangle_{1, s_1}$ will denote a triangle with one vertex being $s_1$ and $\triangle_{2, s_1}$ will denote a pair of triangles that share a common edge with $s_1$ as a vertex of the common edge.

\item $\triangle_{1, s_1 s_2}$ will denote a triangle with an edge $(s_1, s_2)$ and $\triangle_{2, s_1 s_2}$ will denote a pair of triangles with the common edge $(s_1, s_2)$.

\item $H_{1,s_1 s_2 s_3}$ will denote a graph with two different triangles with the edge $(s_1, s_3)$ and two different triangles with the edge $(s_2, s_3)$, $H_{2,s_1 s_2 s_3}$ will denote a graph with two (or respectively one) different triangle(s) with the edge $(s_1, s_3)$ and one (or respectively two) triangle(s) with the edge $(s_2, s_3)$, and $H_{3,s_1 s_2 s_3}$ will denote a graph with one triangle with the edge $(s_1, s_3)$ and one triangle with the edge $(s_2, s_3)$. 

\end{itemize}

For each vertex $s_1$, let $d_{s_1}= N(\triangle_{1, s_1}, G_n)+N(\triangle_{2, s_1}, G_n)$.\footnote{Here, $ N(\triangle_{1, s_1}, G_n)$ denotes the number of triangles in $G_n$ with one vertex in $s_1$. Similarly, $N(\triangle_{2, s_2}, G_n)$ counts the number of 2-pyramids $\triangle_2$ in $G_n$ with one vertex in $s_1$. The notations $N(\triangle_{1, s_1 s_2}, G_n)$, $N(\triangle_{2, s_1 s_2}, G_n)$, $N(H_{1, s_1 s_2  s_3}, G_n)$, $N(H_{2, s_1 s_2  s_3}, G_n)$, and $N(H_{3, s_1 s_2  s_3}, G_n)$ are defined similarly. }
We order the vertices of $G_n$ such that $d_1\geq \dots \geq d_{|V(G_n)|}$. Note that given $1 \leq s_1 < s_2 < s_3 \leq |V(G_n)|$, $d_{s_1 s_3}^2 d_{s_2 s_3}^2$ counts all possible combinations of two ordered triangles with $(s_1, s_3)$ as an edge and two ordered triangles with $(s_2, s_3)$ as an edge. These four triangles can form $H_{1, s_1  s_2 s_3}$, $H_{2, s_1  s_2 s_3}$ or $H_{3, s_1  s_2 s_3}$, each of which is counted only finitely many times in $s(G_n)$. Therefore,
\begin{align}
\sum_{1 \leq s_1 < s_2 < s_3 \leq |V(G_n)|} & d_{s_1 s_3}^2 d_{s_2 s_3}^2 \nonumber \\ 
& \lesssim \sum_{1 \leq s_1 < s_2 < s_3 \leq |V(G_n)|} \Big\{ N(H_{1,s_1 s_2 s_3}, G_n)+N(H_{2,s_1 s_2 s_3}, G_n)+N(H_{3,s_1 s_2 s_3}, G_n) \Big\}. \nonumber 
\end{align}
For $s_1 \in \{1, 2, \ldots, |V(G_n)|\}$ fixed, note that 
\begin{align}\label{eq:NG12_I}
\sum_{s_2, s_3:  s_1 < s_2 < s_3} &\Big\{N(H_{1,s_1 s_2 s_3}, G_n)+N(H_{2,s_1 s_2 s_3}, G_n) +N(H_{3,s_1 s_2 s_3}, G_n)\Big\} \nonumber \\
 \lesssim & \sum_{s_3:  s_1 < s_3} (N(\triangle_{1, s_1 s_3}, G_n)+N(\triangle_{2, s_1 s_3}, G_n)) d_{s_3}. 
\end{align} 
This is obtained by first fixing $s_3 > s_1$ and then choosing the copy of $\triangle_1$ or $\triangle_2$ on the edge $(s_1, s_3)$ in at most $N(\triangle_{1, s_1 s_3}, G_n)+N(\triangle_{2, s_1 s_3}, G_n)$ ways and then choosing the copy of $\triangle_1$ or $\triangle_2$ on the edge $(s_3, s_2)$,  for some $s_1 < s_2 < s_3$, in at most $d_{s_3}$ ways. Another way to bound the quantity on the LHS above is,  
\begin{align}\label{eq:NG12_II}
\sum_{s_2, s_3:  s_1 < s_2 < s_3} \Big\{ N(H_{1,s_1 s_2 s_3}, G_n)+N(H_{2,s_1 s_2 s_3}, G_n)+N(H_{3,s_1 s_2 s_3}, G_n) \Big\} \lesssim  d_{s_1}^2. 
\end{align} 
This is by the product rule of first counting the number $\triangle_{1,s_1}$ and $\triangle_{2,s_1}$ in $d_{s_1}$ ways and then counting the number of $\triangle_1$ and $\triangle_2$ passing through $s_3$ in at most $d_{s_3}\leq d_{s_1}$ ways by the ordering of the vertices. Therefore, combining \eqref{eq:NG12_I} and \eqref{eq:NG12_II} gives, 
\begin{align}\label{eq:G12}
\sum_{1 \leq s_1 < s_2 < s_3 \leq |V(G_n)|} & \left\{ N(H_{1,s_1 s_2 s_3}, G_n)+N(H_{2,s_1 s_2 s_3}, G_n) +N(H_{3,s_1 s_2 s_3}, G_n)\right\}  \nonumber \\ 
& \lesssim \sum_{s_1=1}^{|V(G_n)|} \left\{ \sqrt{d_{s_1}^2} \left(\sum_{s_3=1}^{|V(G_n)|} d_{s_3} (N(\triangle_{1, s_1 s_3}, G_n)+N(\triangle_{2, s_1 s_3}, G_n)) \right)^{\frac{1}{2}} \right\} \nonumber \\
& \lesssim \left(N(\triangle_1, G_n)+N(\triangle_2, G_n)\right)^{\frac{3}{2}} \left(1+N(\triangle_4, G_n) \right)^{\frac{1}{4}}, 
\end{align}
where the last step uses $\sum_{s_1} d_{s_1} \lesssim (N(\triangle_1, G_n)+N(\triangle_2, G_n))$ and $\max_{s_1, s_3}(N(\triangle_{1, s_1 s_3}, G_n)+N(\triangle_{2, s_1 s_3}, G_n)) \lesssim  (1+N(\triangle_4, G_n))^{1/2}$.
The lemma follows from \eq{eq:G12}.  
\end{proof}

We now proceed to bound the terms $A$ and $B$ as defined in \eqref{eq:AB}. 
We may use the properties (a) and (b) of the Hoeffding decomposition implicitly below. First, we bound $A$. To this end, note that
\begin{align}\label{eq:A}
A=\frac{1}{\sigma^4} \sum_{t=1}^{|V(G_n)|} \E\left (\sum_{s: s< t} Y_{s t} +  \sum_{s, s': s < s' < t} Y_{s s' t} \right)^4 \lesssim A_1+A_2,
\end{align}
where
\begin{align*} 
A_1:=\frac{1}{\sigma^4} \sum_{t=1}^{|V(G_n)|} \E\left(\sum_{s: s < t} Y_{s t}\right)^4 \quad \text{and} \quad 
A_2:=\frac{1}{\sigma^4} \sum_{t=1}^{|V(G_n)|} \E\left(\sum_{s, s': s < s' < t} Y_{s s' t}\right)^4.
\end{align*}
By ruling out all the zero-expectation terms, we have $A_1=A_{11}+2 A_{12}$, where
\begin{align}\label{eq:A11}
A_{11}:=\frac{1}{\sigma^4} \sum_{1 \leq s < t \leq |V(G_n)|}  \E (Y_{s t}^4) & \leq \frac{1}{\sigma^4} \sum_{1 \leq s < t \leq |V(G_n)|}  \frac{d_{s t}^4}{c^4} \E \left( \bm 1\{X_{s_1}=X_{s_2}\}-\frac{1}{c} \right)^4 \nonumber \\ 
& \lesssim_c \frac{1}{\sigma^4 } \sum_{1 \leq s < t \leq |V(G_n)|}  \left( d_{st} + {d_{s t} \choose 4} \right) \nonumber \\ 
& \lesssim \frac{N(\triangle_1, G_n)+N(\triangle_4, G_n)}{\sigma^4 }, 
\end{align} 
and
\begin{align}\label{eq:A12}
A_{12} & :=\frac{1}{\sigma^4} \sum_{1\leq s_1 < s_2 < s_3 \leq |V(G_n)|} \E (Y_{s_1 s_3}^2 Y_{s_2 s_3}^2) \nonumber \\ 
& \leq \frac{1}{\sigma^4} \sum_{1\leq s_1 < s_2 < s_3 \leq |V(G_n)|}  \frac{d_{s_1 s_3}^2 d_{s_2 s_3}^2}{c^4} \E \left[\left( \bm 1\{X_{s_1}=X_{s_3}\}-\frac{1}{c} \right)^2 \left( \bm 1\{X_{s_2}=X_{s_3}\}-\frac{1}{c} \right)^2 \right] \nonumber \\ 
& \lesssim_c \frac{s(G_n) }{\sigma^4 } .  
\end{align} 
Next, we consider $A_2$. It can be divided into three non-zero terms without any isolated vertex depending on the number of vertices involved being three or four or five. More specifically,
\begin{align}
A_2\lesssim & \frac{1}{\sigma^4} \sum_{t=1}^{|V(G_n)|} \sum_{s_1, s_2<t:\atop s_1\ne s_2} \E Y_{s_1 s_2 t}^4 \nonumber \\
&+\frac{1}{\sigma^4} \sum_{t=1}^{|V(G_n)|} \sum_{s_1, s_2, s_3<t:\atop |\{s_1, s_2, s_3\}|=3} \E\left( Y_{s_1s_2t}^2 Y_{s_1s_3t}^2+ Y_{s_1s_2t}^2 Y_{s_1s_3t}Y_{s_2s_3t} \right) \nonumber \\
&+ \frac{1}{\sigma^4} \sum_{t=1}^{|V(G_n)|} \sum_{s_1, s_2, s_3, s_4<t:\atop |\{s_1, s_2, s_3, s_4\}|=4} \E \left( Y_{s_1s_2t}^2 Y_{s_3 s_4t}^2+Y_{s_1s_2t} Y_{s_2s_3t} Y_{s_3s_4t} Y_{s_4s_1t}     \right)\label{eq:f002}
\end{align}
These terms are bounded by, up to a polynomial dependence on $\frac{1}{c}$, 
\begin{align}\label{eq:f003} 
O_c\left(\frac{N(\triangle_1, G_n)}{\sigma^4}  +\frac{N(\triangle_2, G_n) +N(H_{23}, G_n)}{\sigma^4} +\frac{ s(G_n) +N(H_{11}, G_n)}{\sigma^4} \right )
\end{align}
where $H_{23}$ and $H_{11}$ are subgraphs shown in Figure~\ref{fig2}.\footnote{For positive sequences $\{a_n\}_{n\geq 1}$ and $\{b_n\}_{n\geq 1}$, $a_n =O_{\square}(b_n)$ means $a_n \leq C  b_n$, for all $n$ large enough, where $C=C(\square) > 0$ is a constant depending on the subscripted parameters.} Note that
\begin{equation}\label{eq:f004}
N(H_{23}, G_n)+N(H_{11}, G_n)\lesssim b(G_n),
\end{equation}
and $N(\triangle_1, G_n)\lesssim s(G_n)$ and $N(\triangle_2, G_n)\lesssim b(G_n)$. Therefore,
\begin{align}\label{eq:A2}
A_2\lesssim_c \frac{ s(G_n) + b(G_n) }{\sigma^4} . 
\end{align} 
This implies, by \eqref{eq:A}, \eqref{eq:A11}, \eqref{eq:A12}, and \eqref{eq:A2}, 
\begin{align}\label{eq:A_bound}
A & \lesssim_c  \frac{N(\triangle_4, G_n) + s(G_n) + b(G_n) }{\sigma^4} . 
\end{align}

Now we turn to bounding $B$. We have
\begin{align}\label{eq:B_terms}
B=& \frac{1}{\sigma^4}\Var\left( \sum_{t=1}^{|V(G_n)|} \left(\sum_{s: s < t} Y_{s t}+ \sum_{s, s': s < s' < t} Y_{s s' t} \right)^2    \right)  \nonumber \\
& \lesssim B_1 + B_2 +  B_3, 
\end{align} 
where 
\begin{align}\label{eq:B123}
B_1 & = \frac{1}{\sigma^4}\Var\left(\sum_{t=1}^{|V(G_n)|} \left(\sum_{s: s < t} Y_{s  t} \right)^2 \right) \nonumber \\ 
B_2 & =\frac{1}{\sigma^4}\Var \left(\sum_{t=1}^{|V(G_n)|} \left(\sum_{s, s': s < s' < t} Y_{s s' t}\right)^2 \right) \nonumber \\  
B_3 & = \frac{1}{\sigma^4}\Var \left( \sum_{t=1}^{|V(G_n)|} \left(\sum_{s: s< t} Y_{s t} \right)\left(\sum_{s, s': s < s' < t} Y_{s s' t} \right) \right)
\end{align} 
We begin with $B_1$. Expanding the square gives, 
\begin{align}\label{eq:B1} 
B_1 & \lesssim \frac{1}{\sigma^4}\Var\left(\sum_{1 \leq s_1 < s_2 \leq |V(G_n)| }  Y_{s_1 s_2}^2 \right) + \frac{1}{\sigma^4} \Var \left(\sum_{1 \leq  s_1 < s_2 < s_3  \leq |V(G_n)|  } Y_{s_1 s_3} Y_{s_2 s_3}\right)   \nonumber \\  
& := B_{11} + B_{12}.
\end{align}
Note that, since $\Cov(Y_{s_1 s_2}^2,  Y_{s_1 s_3}^2)=0$, for $s_1 < s_2 < s_3$, 
\begin{align}\label{eq:B11}
B_{11}= \frac{1}{\sigma^4}\Var\left(\sum_{1 \leq s_1 < s_2 \leq |V(G_n)| }  Y_{s_1 s_2}^2 \right) & = \frac{1}{\sigma^4} \sum_{1 \leq s_1 < s_2 \leq |V(G_n)| } \Var( Y_{s_1 s_2}^2 )   \nonumber \\ 
& \leq \frac{1}{\sigma^4} \sum_{1 \leq s_1 < s_2 \leq |V(G_n)| } \E( Y_{s_1 s_2}^4 ) \nonumber \\ 
& \lesssim_c  \frac{1}{\sigma^4} \left( N(\triangle_1, G_n) + N(\triangle_4, G_n) \right), 
\end{align} 
by \eqref{eq:A11}. Similarly, by ruling out all the zero-expectation terms whenever there is a free index, we have
\begin{align}\label{eq:B12}
B_{12}=&\frac{1}{\sigma^4} \Var\left(\sum_{1 \leq s_1 < s_2 < s_3 \leq |V(G_n)|} Y_{s_1 s_3} Y_{s_2 s_3} \right) \nonumber \\
=& \frac{1}{\sigma^4} \E \left(\sum_{1 \leq s_1 < s_2 < s_3 \leq |V(G_n)|} Y_{s_1 s_3} Y_{s_2 s_3}\right)^2  \nonumber  \\
=& \frac{1}{\sigma^4} \E \left(\sum_{1 \leq s_1 < s_2 < s_3 \leq |V(G_n)|} Y_{s_1 s_3}^2 Y_{s_2 s_3}^2 \right) +\frac{2}{\sigma^4} \E \left(\sum_{s_1 < s_2 < s_3< s_4} Y_{s_1 s_3} Y_{s_2 s_3}Y_{s_1 s_4} Y_{s_2 s_4} \right)  \nonumber  \\ 
\lesssim_c & \frac{s(G_n) }{\sigma^4 } + \frac{1}{\sigma^4 }\sum_{1\leq s_1 < s_2 < s_3<s_4 \leq |V(G_n)|} d_{s_1 s_3} d_{s_2 s_3} d_{s_1 s_4} d_{s_2 s_4}  \nonumber  \\ 
= & \frac{ s(G_n)  +  b(G_n) }{\sigma^4}. 
\end{align}
Hence, using \eqref{eq:B11} and \eqref{eq:B12} in \eqref{eq:B1} gives (recall the bound $N(\triangle_1, G_n)\lesssim s(G_n)$), 
\begin{align}\label{eq:B1_bound} 
B_1 & \lesssim_c  \frac{ N(\triangle_4, G_n) + s(G_n) + b(G_n) }{\sigma^4} .  
\end{align}

Now, we bound $B_2$. Recalling the definition of $B_2$ from \eqref{eq:B123} gives, 
\begin{align}\label{eq:B2123}
B_2 
=& \frac{1}{\sigma^4}\Var \bigg(\sum_{1 \leq s_1 < s_2 < s_3 \leq |V(G_n)|} Y_{s_1 s_2 s_3}^2 + \sum_{s_3=1}^{|V(G_n)|} \sum_{s_1, s_2 : s_1 < s_2 < s_3} \sum_{s_4: s_4 < s_3 \atop s_4\notin \{s_1, s_2\}} Y_{s_1 s_2 s_3} (Y_{s_1 s_4 s_3}+Y_{s_4 s_2 s_3}) \nonumber \\
& \qquad \qquad + \sum_{s_3=1}^{|V(G_n)|} \sum_{s_1, s_2 : s_1 < s_2 < s_3} \sum_{s_4, s_5: s_4<s_5<s_3 \atop |\{s_1, s_2, s_4, s_5\}|=4}  Y_{s_1 s_2 s_3} Y_{s_4 s_5 s_3}  \bigg) \nonumber \\ 
\lesssim & B_{21} + B_{22} + B_{23},
\end{align}
where $B_{21}$, $B_{22}$, and $B_{23}$ are the respective variances of the 3 terms above.  Note that, since the covariances are non-zero only when there are at least two common vertices, 
\begin{align}\label{eq:B21}
B_{21}:=\frac{1}{\sigma^4}\Var\left( \sum_{1 \leq s_1 < s_2 < s_3 \leq |V(G_n)|} Y_{s_1 s_2 s_3}^2  \right) \lesssim_c \frac{ N(\triangle_1, G_n) + N(\triangle_2, G_n) }{\sigma^4} . 
\end{align} 
Next, for $B_{22}$ the covariances of the summands can be divided into three terms depending on the number of common vertices being two, three, or four. 
Following similar calculations as in \eq{eq:f002}--\eq{eq:f004}, these terms are bounded by, up to a polynomial dependence on $\frac{1}{c}$
$$B_{22} \lesssim_c \frac{1}{\sigma^4}\left( N(H_{25}, G_n)+N(H_{11}, G_n)+N(H_{23}, G_n) +N(\triangle_2, G_n) \right)\lesssim \frac{b(G_n)}{\sigma^4},$$
where subgraphs $H_{25}$, $H_{11}$ and $H_{23}$ are as shown in Figure~\ref{fig2}. Similarly, for $B_{23}$ the covariances of the summands can be divided into two terms depending on the number of common vertices being four or five and 
$$B_{23} \lesssim_c  \frac{ s(G_n) + b(G_n) }{ \sigma^4 }. $$ Now, recalling \eqref{eq:B2123}, $N(\triangle_1, G_n)\lesssim s(G_n)$ and $N(\triangle_2, G_n)\lesssim b(G_n)$, and combining the bounds for $B_{21}$, $B_{22}$, and $B_{23}$, gives   
\begin{align}\label{eq:B2_bound} 
B_2 & \lesssim_c  \frac{s(G_n) + b(G_n)}{\sigma^4} .  
\end{align} 

Finally, we bound $B_{3}$. Recalling the definition of $B_3$ from \eqref{eq:B123} gives, 
\begin{align}\label{eq:B3}
B_3 & = \frac{1}{\sigma^4}\Var\left( \sum_{1 \leq s_1 < s_2 < s_3 \leq |V(G_n)| } (Y_{s_1 s_3}+Y_{s_2 s_3}) Y_{s_1s_2 s_3}+\sum_{1 \leq s_1 < s_2 < s_3 \leq |V(G_n)| }\sum_{s_4: s_4<s_3 \atop s_4 \notin \{s_1, s_2\}} Y_{s_4 s_3} Y_{s_1s_2 s_3}  \right) \nonumber \\ 
& \lesssim B_{31} + B_{31}' + B_{32} + B_{32}' + B_{32}'', 
\end{align}
where 
\begin{align}\label{eq:B3_terms}
B_{31} & = \frac{1}{\sigma^4}\Var \left( \sum_{1 \leq s_1 < s_2 < s_3 \leq |V(G_n)| } Y_{s_1 s_3} Y_{s_1 s_2 s_3} \right), \nonumber \\ 
B_{31}' & =\frac{1}{\sigma^4}\Var \left( \sum_{1 \leq 1 \leq s_1 < s_2 < s_3 \leq |V(G_n)| } Y_{s_2 s_3} Y_{s_1 s_2 s_3}  \right), \nonumber \\
B_{32} & =\frac{1}{\sigma^4}\Var \left( \sum_{  1 \leq s_1 < s_2 < s_3 \leq |V(G_n)|  }\sum_{s_4:  s_4< s_1} Y_{s_4 s_3} Y_{s_1 s_2 s_3} \right), \nonumber \\ 
B_{32}' & =\frac{1}{\sigma^4}\Var \left( \sum_{  1 \leq s_1 < s_2 < s_3 \leq |V(G_n)| }\sum_{s_4: s_1 < s_4 < s_2} Y_{s_4 s_3} Y_{s_1 s_2 s_3} \right), \nonumber \\ 
B_{32}'' & =\frac{1}{\sigma^4}\Var \left( \sum_{ 1 \leq s_1 < s_2 < s_3 \leq |V(G_n)| }\sum_{s_4: s_2 < s_4 < s_3} Y_{s_4 s_3} Y_{s_1 s_2 s_3} \right). 
\end{align}
We begin with $B_{31}$. In this the covariances of the summands can be divided into two terms depending on the number of common vertices being two or three.
Following similar calculations as in \eq{eq:f002}--\eq{eq:f004}, these terms are bounded respectively by $O_c(b(G_n))$ and $O_c(N(\triangle_1, G_n)+N(\triangle_2, G_n)+N(\triangle_3, G_n))$.
Therefore,
$$B_{31}\lesssim_c \frac{N(\triangle_1, G_n)+N(\triangle_4, G_n)+ b(G_n)}{\sigma^4 }.$$ 
Similarly, 
$$B_{31}'\lesssim_c \frac{N(\triangle_1, G_n)+ N(\triangle_4, G_n)+ b(G_n)}{\sigma^4 }  .$$ 
Finally, in $B_{32}$ the covariances of the summands can be divided into two terms depending on the number of common vertices being three or four. Again, following similar calculations as in \eq{eq:f002}--\eq{eq:f004}, these terms are bounded respectively by $O_c (b(G_n))$ and $O_c (s(G_n))$. Therefore, $$B_{32} \lesssim_c \frac{s(G_n) + b(G_n)}{\sigma^4}. $$ 
Similarly, $\max\{B_{32}', B_{32}'' \} \lesssim_c  \frac{1}{\sigma^4}(s(G_n) + b(G_n))$.  
Hence, combining the estimates above with \eqref{eq:B3} and \eqref{eq:B3_terms} gives, 
\begin{align}\label{eq:B3_bound} 
B_3 & \lesssim_c  \frac{N(\triangle_1, G_n) + N(\triangle_4, G_n) + s(G_n) + b(G_n)}{\sigma^4}.  
\end{align} 

Finally, recalling \eqref{eq:B_terms}, and combining the bounds in \eqref{eq:B1_bound}, \eqref{eq:B2_bound}, \eqref{eq:B3_bound} gives (recall $N(\triangle_1, G_n)\lesssim s(G_n)$), 
\begin{align*}
B & \lesssim_c  \frac{N(\triangle_4, G_n) + s(G_n) + b(G_n) }{\sigma^4} .  
\end{align*}
Moreover, by \eqref{eq:A_bound} and using $\sigma^4\asymp_c  \left( N(\triangle_1, G_n) + N(\triangle_2, G_n) \right)^2$, 
\begin{align}\label{eq:Z3bound_I} 
A+B & \lesssim_c  \frac{ N(\triangle_4, G_n) + s(G_n) + b(G_n)}{\sigma^4}    \lesssim_c  \frac{ N(\triangle_4, G_n) + s(G_n)+b(G_n)}{ \left( N(\triangle_1, G_n)  + N(\triangle_2, G_n) \right)^2  } . 
\end{align}
Bounding $s(G_n)$ by Lemma~\ref{l1} gives, 
\begin{align}\label{eq:Z3bound_II}
\frac{ s(G_n) }{  \left(  N(\triangle_1, G_n) + N(\triangle_2, G_n)  \right)^2 } & \lesssim \left( \frac{ 1 +  N(\triangle_4, G_n) }{ \left(  N(\triangle_1, G_n) + N(\triangle_2, G_n)  \right)^2} \right)^{\frac{1}{4}}. 
\end{align} 
The result in \eq{12} now follows by using \eqref{eq:Z3bound_I} and \eqref{eq:Z3bound_II} in  \eq{1}. \hfill $\Box$ \\

Having completed the proof of Theorem \ref{T1}, we now construct two examples which show the necessity of the error terms $R_1$ and $R_2$ (recall definitions from \eqref{eq:R12}) in \eqref{12}, that is, if these error  terms are non-vanishing then a CLT for $Z_3(G_n)$ might not hold.

\begin{example}\label{example1}
Consider the $n$-pyramid $\triangle_{n}$ with vertex set $V(\triangle_n)=\{1, 2, u_1, u_2, \ldots, u_n\}$, where the vertices $\{1,2,u_s\}$ form triangles, for $1\leq s\leq n$.
Note that this graph has $|V(\triangle_n)|=n+2$ vertices, $|E(\triangle_n)|=2n+1$ edges, and $N(K_3, \triangle_n)=n$ triangles.
Moreover, note that $N(\triangle_2, \triangle_n)={n \choose 2}\asymp n^2$, $N(\triangle_4, \triangle_n)={n \choose 4}\asymp n^4$, and $b(\triangle_n)={n \choose 2}\asymp n^2$.
Therefore, recalling \eqref{eq:R12},
$$R_1 = \frac{1+N(\triangle_4, \triangle_n)}{( N(K_3, \triangle_n) + N(\triangle_2, \triangle_n))^2}  \asymp 1 \quad \text{ and } \quad R_2 = \frac{b(\triangle_n)}{(  N(K_3, \triangle_n) + N(\triangle_2, \triangle_n))^2} \asymp \frac{1}{n^2}\to 0,$$
that is, the error term in Theorem \ref{T1} does not vanish, since $R_1$ does not go to zero. In fact, in this case, $Z_3(\triangle_n)$ has a non-normal limit, as argued below. Note that with probability $\frac{1}{c}$, the vertices 1 and 2 have the same color. In this case, $T_3(\triangle_n)\sim \mathrm{Bin}(n, \frac{1}{c})$ and
\begin{align}\label{eq:example1_I}
\frac{T_3(\triangle_n)-\frac{n}{c^2}}{n}\stackrel{P} \rightarrow \frac{1}{c} \left(1-\frac{1}{c} \right).
\end{align}
On the other hand, with probability $1-\frac{1}{c}$, the vertices 1 and 2 have different colors. In this case, $T_3(\triangle_n)=0$ and
\begin{align}\label{eq:example1_II}
\frac{T_3(\triangle_n)-\frac{n}{c^2}}{n}\stackrel{P} \rightarrow -\frac{1}{c^2}.
\end{align}
Therefore, combining \eqref{eq:example1_I} and \eqref{eq:example1_II}, 
\begin{equation}\label{eq:f001}
\frac{T_3(\triangle_n)-\frac{n}{c^2}}{n}\stackrel{D} \rightarrow \frac{1}{c}\delta_{ \frac{1}{c} \left(1-\frac{1}{c} \right)} + \left(1-\frac{1}{c} \right) \delta_{-\frac{1}{c^2}},
\end{equation}
which is a 2-point discrete distribution. 
\end{example}

Note that in the example above the term $R_1$ is non-vanishing. We now construct a sequence of graphs which has a non-normal limiting distribution for which the term $R_2$ in \eqref{eq:R12} is non-vanishing.

\begin{figure}[h]
\centering
\begin{minipage}[c]{1.0\textwidth}
\centering
\includegraphics[width=1.85in]
    {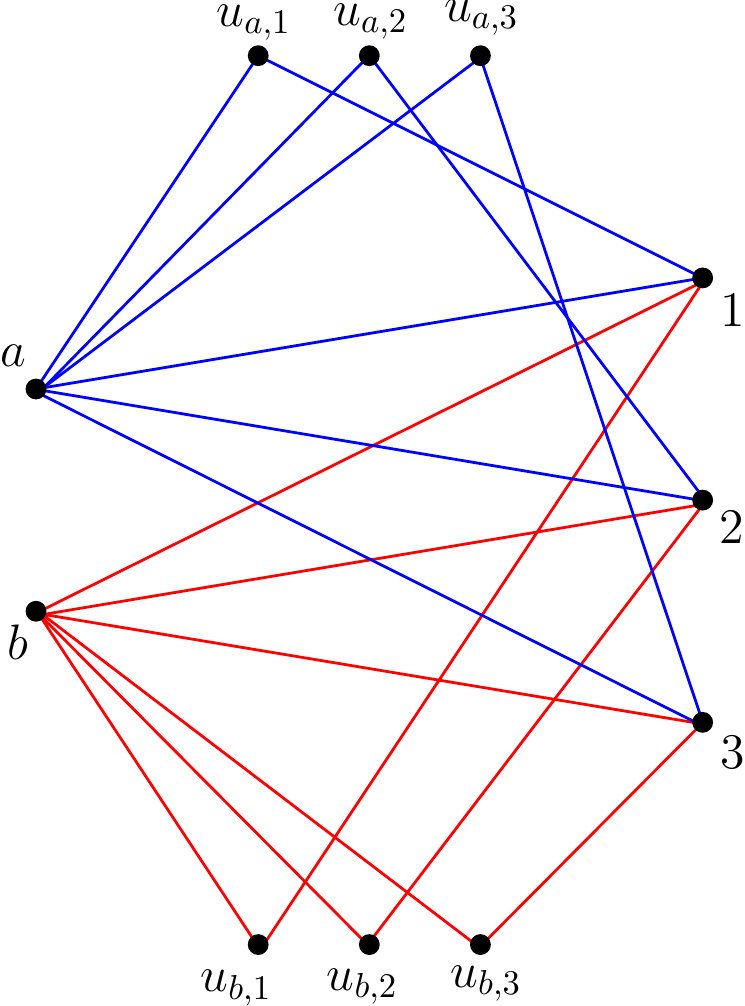}\\
\end{minipage}
\caption{\small{The graph $\cB_3^\triangle$ as defined in Example \ref{example2}.}}
\label{fig:example_moment_II}
\end{figure}

\begin{example}\label{example2}
Consider the graph $G_n$ with vertex set $V(G_n)=V_1 \bigcup V_2 \bigcup V_3 \bigcup V_4$, where 
$$V_1=\{a, b\}, ~ V_2=\{1, 2,  \ldots n\}, ~ V_3=\{u_{a, 1}, u_{a, 2}, \ldots, u_{a, n}\}, \text{ and }  V_4=\{u_{b, 1}, u_{b, 2}, \ldots, u_{b, n}\},$$ 
where the vertices $\{a, s, u_{a, s}\}$ form a triangle and the vertices  $\{b, s, u_{b, s}\}$ form a triangle, for every $1 \leq s \leq n$.  We denote this graph by $\cB_n^\triangle$. (The graph  $\cB_3^\triangle$ is shown in Figure \ref{fig:example_moment_II}(b).) Note that this graph has $|V(\cB_n^\triangle)|=3n+2$ vertices, $|E(\cB_n^\triangle)|=6n$ edges, and $N(K_3, \cB_n^\triangle)=2n$ triangles (one on every edge between $V_1$ and $V_2$). Moreover, note that $N(\triangle_2, \cB_n^\triangle)=0$, hence, $N(\triangle_4, G_n)=0$. Therefore, $R_1 \asymp 1/n \rightarrow 0$.
On other hand, it is easy to check that $b(G_n) \asymp n^2$, since,  
for $1 \leq s < t \leq n$, $d_{as} = d_{b s}=d_{bt}=d_{at}=1$. Therefore, 
the second term in \eqref{12} is
$$R_2=\frac{b(G_n)}{ ( N(K_3, G_n)+N(\triangle_2, G_n))^2} \asymp 1,$$
Note that this does not go  to zero. 
In fact, $Z_3(\cB_n^\triangle)$ has a non-normal limit,  which we directly derive below. 

Let $A^+$ be the event that the vertices $a$ and $b$ have the same color, and $A^-$ the event that $a$ and $b$ have different colors. Denote the number of monochromatic edges in the complete bipartite formed between $V_1$ and $V_2$ as $T_{12}$.  On the event $A^+$,  $T_{12} \sim 2 \mathrm{Bin}(n, 1/c)$. Therefore, $$T_3(\cB_n^\triangle) \sim \mathrm{Bin}(T_{12}, 1/c).$$ Then, using 
$$\frac{T_3(\cB_n^\triangle) - \frac{T_{12}}{c} }{\sqrt{ T_{12} }} \Big|T_{12} \stackrel{D}\rightarrow N\left(0,  \frac{1}{c} \left( 1 - \frac{1}{c} \right)\right), \quad \frac{ \frac{T_{12}}{c} - \frac{2n}{c^2} }{ \sqrt{n} } \stackrel{D}\rightarrow N\left(0,  \frac{4}{c^3} \left( 1 - \frac{1}{c} \right)\right),$$
and $T_{12}/n  \stackrel{P}\rightarrow 2/c$, it follows that, on the event $A^+$, 
\begin{align}\label{eq:graph_B_I}
\frac{T_3(\cB_n^\triangle) - \frac{2n}{c^2} }{\sqrt{ n }} \stackrel{D}\rightarrow N\left(0,    \left(\frac{4}{c^3} + \frac{2}{c^2} \right) \left( 1 - \frac{1}{c} \right) \right).
\end{align}

Next, consider the event $A^-$. Assume, without also generality, the vertex $a$ is colored with color $1$ and the vertex $b$ is colored with color 2. Then $T_{12} \sim N_1+N_2$, where $(N_1, N_2, \ldots, N_c) \sim \mathrm{Multi}(n, 1/c, 1/c, \ldots, 1/c)$, where $N_s$ denotes the number of vertices of color $s$ in the set $V_2$. As before, $T_3(\cB_n^\triangle) \sim \mathrm{Bin}(T_{12}, 1/c)$. Now, note that, 
$$\frac{T_3(\cB_n^\triangle) - \frac{T_{12}}{c}}{\sqrt{T_{12}}}  \Big|T_{12}   \stackrel{D}\rightarrow N\left(0,  \frac{1}{c} \left( 1 - \frac{1}{c} \right)\right), \quad \frac{\frac{T_{12}}{c} - \frac{2n}{c^2}}{\sqrt{n}}  \stackrel{D}\rightarrow N\left(0,  \frac{2}{c^3} \left( 1 - \frac{2}{c} \right)\right),$$
and $T_{12}/n \stackrel{P}\rightarrow  2/c$, since $N_1+N_2 \sim \mathrm{Bin}(n, 2/c)$. Therefore, on the event $A^-$, 
\begin{align}\label{eq:graph_B_II}
\frac{T_3(\cB_n^\triangle) - \frac{2n}{c^2} }{\sqrt{ n }} \stackrel{D}\rightarrow N\left(0,    \frac{2}{c^2} \left( 1 - \frac{2}{c^2} \right) \right). 
\end{align} 
Combining \eqref{eq:graph_B_I} and \eqref{eq:graph_B_II}, and noting that $\P(A^+)=1/c$ and $\P(A^-)=1-1/c$, we get 
\begin{align}\label{eq:limit_B}
& \frac{T_3(\cB_n^\triangle) - \frac{2n}{c^2} }{\sqrt{ n }} \nonumber \\ 
& \quad \stackrel{D}\rightarrow \frac{1}{c} \cdot N\left(0,      \left(\frac{4}{c^3} + \frac{2}{c^2} \right) \left( 1 - \frac{1}{c} \right) \right) + \left(1-\frac{1}{c} \right) \cdot N\left(0,     \frac{2}{c^2} \left( 1 - \frac{2}{c^2} \right) \right), 
\end{align}
which is a mixture of two normals. 

\end{example}

\section{Proof of Theorem \ref{T2}}\label{sec4}

In this section we prove Theorem \ref{T2}. The proof of the fourth moment bound in \eqref{13} is given in Section \ref{sec:th4thmomentpf_I} below. The counterexample to the fourth-moment phenomenon for $2 \leq c \leq 4$ is described in Section \ref{sec:th4thmomentpf_II}.

\subsection{Proof of the Fourth Moment Bound for $c\geq 5$}
\label{sec:th4thmomentpf_I}

The main step in the proof of fourth moment error bound in \eqref{13} is the computation of $\E (Z_3(G_n)^4) - 3$. 

\begin{lemma}\label{l2} For any graph sequence $\{G_n\}_{n \geq 1}$ the following hold: 
\begin{itemize} 
\item[(a)] For any $c \geq 2$, 
\ben{\label{14}
\E (Z_3(G_n)^4)-3=\frac{1}{\sigma^4}\left\{\sum_{s=1}^4 \delta_s N(\triangle_s, G_n)+\sum_{s=1}^{28} h_s N(H_s, G_n)   \right\},
}
where the graphs $\triangle_1, \triangle_2, \triangle_3, \triangle_4$, $H_1, H_2, \ldots, H_{28}$, along with the specified triangles, are listed in Figure~\ref{fig2}, and the coefficients $\delta_1, \delta_2, \delta_3, \delta_4$, $h_1, h_2, \ldots, h_{28}$ are given in Table~\ref{table:coefficient_list}.\footnote{Throughout this section, we will use $N(H_s, G_n)$ to count the number of copies of $H_s$ with the specified triangles in $G_n$, for $1 \leq s \leq 28$ (see the caption of Figure~\ref{fig2} for illustration).}

\item[(b)] If $c\geq 5$, all the coefficients $\delta_1, \delta_2, \delta_3, \delta_4$, $h_1, h_2, \ldots, h_{28}$ are positive.
\end{itemize} 
\end{lemma}

\begin{proof} Note that 
$$\E (Z_3(G_n)^4) - 3=\frac{1}{\sigma^4}\left\{ \E \left[ \sum_{s_1<s_2< s_3} a_{s_1 s_2 s_3}(G_n) \bm 1\left\{X_{s_1}=X_{s_2}=X_{s_3})-\frac{1}{c^2} \right\}  \right]^4-3\sigma^4   \right\}, $$ 
where $a_{s_1 s_2 s_3}(G_n) := a_{s_1s_2}(G_n) a_{s_2s_3}(G_n) a_{s_1s_3}(G_n)$. Expanding out the fourth powers above gives a sum over graphs formed by the union of 4 triangles. The idea of the proof is to group terms corresponding to graphs with specified triangles (see the caption of Figure~\ref{fig2} for illustration) and keep track of the corresponding coefficients. The graphs which show up and the corresponding coefficients are listed in Table \ref{table:coefficient_list}. Here, we only show the computation of the coefficient $\delta_4$ which corresponds to the graph $\triangle_4$. The other coefficients are obtained by similar tedious but straightforward computations. The details are omitted.

Suppose in the expansion of 
\begin{align*}
\E\left[ \sum_{1 \leq s_1 < s_2 < s_3 \leq |V(G_n)|} a_{s_1 s_2 s_3}(G_n) \left(\bm 1\{X_{s_1}=X_{s_2}=X_{s_3})-\frac{1}{c^2} \right)\right]^4,
\end{align*} 
the four triangles which form a specified $\triangle_4$ in the graph $G_n$,  have $(1,2)$ as the the common edge of the four triangles and $3,4,5,6$ as the other four vertices of $\triangle_4$. Then, denoting $X_{s_1s_2 s_3}= \{X_{s_1}= X_{s_2} = X_{s_3} \}$, 
\begin{align*}
&\E\left( \bm 1 \{X_{123} \} -\frac{1}{c^2} \right) \left( \bm 1\{ X_{124} \} -\frac{1}{c^2} \right) \left( \bm 1\{X_{125} \} -\frac{1}{c^2} \right) \left( \bm 1\{ X_{126} \} -\frac{1}{c^2} \right)\\
& ~~~~~ = \P(X_1=X_2=X_3=X_4=X_5=X_6)-\frac{4}{c^2}\P(X_1=X_2=X_3=X_4=X_5) \\ 
& ~~~~~~~~~~~~~ +\frac{6}{c^4}\P(X_1=X_2=X_3=X_4) -\frac{4}{c^6} \P(X_1=X_2=X_3)+\frac{1}{c^8}\\
& ~~~~~ = \frac{1}{c^5}-\frac{4}{c^6}+\frac{6}{c^7}-\frac{3}{c^8}.
\end{align*} 
Since there $4!=24$ ways to order the four triangles, the contribution of $\E Z_3(G_n)^4$ to the specified $\triangle_4$ is
\ben{\label{16}
24\left(\frac{1}{c^5}-\frac{4}{c^6}+\frac{6}{c^7}-\frac{3}{c^8}\right).
}
Next, write $N(\triangle_1, G_n)=\sum_{\triangle_1 \sqsubseteq G_n} 1 $ and $N(\triangle_2, G_n)=\sum_{\triangle_2 \sqsubseteq G_n} 1$, where the sum is over distinct subgraphs of $G_n$ isomorphic to $\triangle_1$ and $\triangle_2$, respectively.\footnote{For two graphs $G$ and $H$, $H \sqsubseteq G$ means $H$ is a subgraph of $G$.} Then expanding
$$3\sigma^4=3 \left[   \frac{1}{c^2} \left(1-\frac{1}{c^2} \right) N(\triangle_1, G_n)+ 2\left(\frac{1}{c^3}-\frac{1}{c^4} \right) N(\triangle_2, G_n)  \right]^2,$$ 
we get a sum over graphs obtained by the union of two copies of $\triangle_1$, or two copies of $\triangle_2$, or one copy of $\triangle_1$ and one copy of $\triangle_2$.  Note that only possible way to get a $\triangle_4$ is to have the union of two copies of $\triangle_2$ joined at the base. Since there are ${4 \choose 2}=6$ ways to choosing the two non-base vertices of $\triangle_2$ from the vertices $3, 4, 5, 6$,  the contribution of $3\sigma^4$ to a specified $\triangle_4$ is
\ben{\label{17}
3\times {4\choose 2} \times \left[2\left(\frac{1}{c^3}-\frac{1}{c^4}\right)  \right]^2=24\left(\frac{3}{c^6}-\frac{6}{c^7}+\frac{3}{c^8} \right).
}
By combining \eq{16} and \eq{17}, we obtain  $\delta_4=24\left(\frac{1}{c^5}-\frac{7}{c^6}+\frac{12}{c^7}-\frac{6}{c^8}\right)$. 

The conclusion in part (b) can be verified directly from the expressions in Table \ref{table:coefficient_list}. 
\end{proof}

Using the lemma above we now complete the proof of the first part of Theorem \ref{T2}. 
Note that, for $c\geq 5$, by Lemma~\ref{l2}(b), 
\ben{\label{18}
\E (Z_3(G_n)^4)-3 \gtrsim_c \frac{1}{\sigma^4} \left( \sum_{s=1}^4  N(\triangle_s, G_n) +\sum_{s=1}^{28} N(H_s,G_n) \right) .} 
Now, it is straightforward to check that $b(G_n)  \lesssim \sum_{s=1}^4   N(\triangle_s, G_n)+\sum_{s=1}^{28}   N(H_s,G_n).$ Therefore, \eq{13} follows from \eq{12} and \eq{18}. This completes the proof of the first part of Theorem \ref{T2}.

\subsection{Counterexample to the Fourth Moment Phenomenon for $2 \leq c \leq 4$}
\label{sec:th4thmomentpf_II}

Here, we prove the second part of Theorem \ref{T2}. Suppose  $2 \leq c \leq 4$, and consider $G_n$ to be the disjoint union of the $n$-pyramid $\triangle_{n}$ and $\cB_{n'}^\Delta$ (as defined in Example \ref{example2}), where $n' \asymp n^2$ (will be specified later). Then, noting that $\sigma^4 \asymp N(K_3, G_n)^2 + N(\triangle_2, G_n)^2 \asymp n^4 $, and using Lemma \ref{l2}(a)  gives,  
$$
\E (Z_3(G_n)^4)-3=\frac{1}{\sigma^4}\left\{\sum_{s=1}^4 \delta_s { n \choose s}+2\delta_1 n' +  h_{16}     {n' \choose 2} \right\} = \frac{1}{\sigma^4}\left\{  \delta_4 { n \choose 4}  +  h_{16} {n' \choose 2}  \right\} + o(1) ,$$
since $N(H_{16}, G_n) = {n' \choose 2}$, where the graph $H_{16}$ is given in the appendix in Figure \ref{fig2}. Now, note from Table \ref{table:coefficient_list} that $h_{16}>0$, for all $c \geq 2$, but $\delta_4 < 0$, for $2 \leq c \leq 4$. Therefore, choosing $n'=\lceil \sqrt{\frac{2|\delta_4|}{h_{16}} { n \choose 4}} \rceil  > 0$ ensures that $\E Z_3(G_n)^4 \rightarrow 3$. 

Next, 
since $\E(T_3(G_n))= \E (T_3(\triangle_n)) + \E (T_3(\cB_{n'}^\Delta)) = \frac{n}{c^2} + \frac{2n'}{c^2} $, 
\begin{align}
& \frac{T_3(G_n) - \E (T_3(G_n))}{\sqrt{n'}} = \frac{T_3(\triangle_n)  - \frac{n}{c^2}}{\sqrt{n'}} + \frac{ T_3(\cB_{n'}^\Delta, G_n)    - \frac{2 n'}{c^2}}{\sqrt{n'}}  \stackrel{D} \rightarrow \frac{n}{\sqrt{n'}}I+ J, \nonumber 
\end{align}
where $I$ and $J$ are independent, $I$ is as in \eq{eq:f001}, and $J$ is the mixture of 2 normals as in \eqref{eq:limit_B}. This shows that $Z_3(G_n)$ has a non-normal limit, completing the proof of the second part of Theorem \ref{T2}.

\section{Proof of Theorem \ref{THM:4MOMENT}}
\label{sec:4momentpf}

Recall the definition $T_3(G_n)$ from \eqref{eq:T3Gn}. We rewrite $T_3(G_n)$ as follows: 
$$T_3(G_n):=\frac{1}{6} \sum_{\bm s \in V(G_n)_{3} } a_{s_1 s_2}(G_n) a_{s_2 s_3}(G_n) a_{s_3 s_1}(G_n) \bm 1\{ X_{=\bm s} \},$$
where:  
\begin{itemize}
\item[--] $ V(G_n)_{3}$ is the set of all $3$-tuples ${\bm s}=(s_1, s_2, s_3)\in  V(G_n)^{3}$ with distinct indices.\footnote{For a set $S$, the set $S^N$ denotes the $N$-fold cartesian product $S\times S \times \cdots \times S$.} Thus, the cardinality of $V(G_n)_{3}$ is $\frac{|V(G_n)|!}{(|V(G_n)|-3)!}$. 
\item[--] For any ${\bm s}=(s_1, s_2, s_3) \in  V(G_n)_{3}$, $\bm 1\{X_{=\bm s}\}:= \bm 1\{X_{s_1}=X_{s_2}=X_{s_{3}}\}$. 
\end{itemize}

Next, for $v \in V(G_n)$ and $1 \leq a \leq c$, let $Z_{v}(a)=\bm 1\{X_{v}=a\}-\frac{1}{c}$. Then, for $\bm s =(s_1, s_2, s_3) \in V(G_n)_{3}$,  it is easy to check that 
\begin{align}\label{eq:product_sum}
\bm  1\{X_{=\bm s}\}-\frac{1}{c^2} &= \sum_{a=1}^c \left\{ \bm  1 \{X_{s_1}=X_{s_2}=X_{s_3}=a\} -\frac{1}{c^3} \right\}   \nonumber \\
&=\sum_{a=1}^c \left\{ \frac{Z_{s_1}(a) Z_{s_2}(a) + Z_{s_2}(a) Z_{s_3}(a)+  Z_{s_3}(a) Z_{s_1}(a)}{c}  +  Z_{s_1}(a) Z_{s_2}(a) Z_{s_3}(a) \right\},
\end{align}
using $\sum_{a=1}^c Z_v(a)=0$, for all $v \in V(G_n)$. Next, define $M(\bm s, G_n)= \frac{a_{s_1 s_2}(G_n) a_{s_2 s_3}(G_n) a_{s_3 s_1}(G_n)}{6}$ and let
\begin{align}\label{eq:T1}
T_1(K_3, G_n) & = \frac{1}{c}\sum_{a=1}^c \sum_{\bm s \in V(G_n)_{3} } M(\bm s, G_n) \left\{Z_{s_1}(a) Z_{s_2}(a) + Z_{s_2}(a) Z_{s_3}(a)+  Z_{s_3}(a) Z_{s_1}(a) \right\} \nonumber \\ 
& =  \frac{1}{2c}\sum_{a=1}^c \sum_{1 \leq  u \ne v \leq |V(G_n)|} d_{uv} Z_{u}(a) Z_{v}(a), 
\end{align}
where $d_{u v}$ is the number of triangles in $G_n$ with $(u, v)$ as an edge, and 
\begin{align}\label{eq:T2}
T_2(K_3, G_n)= \sum_{a=1}^c \sum_{\bm s \in V(G_n)_{3} } M(\bm s, G_n) Z_{s_1}(a) Z_{s_2}(a) Z_{s_3}(a).
\end{align}
Note that 
$$Z_3(G_n)=\frac{T_1(K_3, G_n) + T_2(K_3, G_n)}{\sqrt{\Var(T_3(G_n))}}.$$

%

With these definitions, we will show that the moments of $Z_3(G_n)$ are bounded. To this end, for a hypergraph $F=(V(F), E(F))$, the degree of a vertex $x \in V(F)$, to be denoted by $d_F(x)$, is the number of hyperedges passing through on $x$. Let $d_{\min}(F)$ denote the minimum degree of the hypergraph.

\begin{lemma}\label{lm:Z_moment} For any integer $r \geq 1$, $\E(|Z_3(G_n)|^r) < C(c, r)$, where $C(c, r)$ is a constant depending only $c$ and $r$. 
\end{lemma}

The result in Theorem \ref{THM:4MOMENT} is an immediate consequence of this lemma: In particular, if $Z_3(G_n)  \stackrel{D} \rightarrow N(0, 1)$, then, by the above lemma and uniform integrability, $\E(Z_3(G_n)^4) \rightarrow 3$. \\

\noindent{\it Proof of Lemma} \ref{lm:Z_moment}: Note that it suffices to prove the result for $r$ even. Moreover, recalling \eqref{eq:T1} and \eqref{eq:T2}, and by the binomial expansion, it suffices to show the following: For all $r \geq 2$ even, 
\begin{align}\label{eq:moment_Z3_I}
\frac{\E(T_1(K_3, G_n)^r)}{\Var(T_3(G_n))^{\frac{r}{2}}} \lesssim_{c, r} 1 \quad \text{and} \quad \frac{\E(T_2(K_3, G_n)^r)}{\Var(T_3(G_n))^{\frac{r}{2}}} \lesssim_{c, r} 1. 
\end{align}

We begin with $T_2(K_3, G_n)$. For $\bm s=(s_1, s_2, s_3) \in V(G_n)_3$, denote by $Z_{\bm s}(a)=Z_{s_1}(a) Z_{s_2}(a)  Z_{s_3}(a)$. Then a direct expansion gives,  
\begin{align*} 
\frac{\E(T_2(K_3, G_n)^r)}{\Var(T_3(G_n))^{\frac{r}{2}}} = &\frac{1}{\Var(T_3(G_n))^{\frac{r}{2}}} \sum_{a_1,\cdots, a_r \in [c]} \sum_{\bm s_1, \bm s_2, \ldots, \bm s_r  \in V(G_n)_3}  \prod_{j=1}^r M(\bm s_j, G_n)  \E \left( \prod_{j=1}^r Z_{\bm s_j}(a_j) \right).
\end{align*}
Let $F$ be the 3-uniform multi-hypergraph formed by the union of the hyperedges $\bm s_1, \bm s_2, \ldots, \bm s_r$. Note that 
$$ \E \left( \prod_{j=1}^r Z_{\bm s_j}(a_j) \right) = 0, $$ 
whenever there exists a $v \in V(F)$ with $d_F(v)=1$. This implies, 
\begin{align}\label{eq:T2moment} 
 \frac{\E(T_2(K_3, G_n)^r)}{\Var(T_3(G_n))^{\frac{r}{2}}}  \lesssim_{c, r} \frac{1}{\Var(T_3(G_n))^{\frac{r}{2}}}  \sum_{F \in \cH_r: d_{\mathrm{min}}(F) \geq 2}   \sum_{\bm s \in V(G_n)_{V(F)}} \prod_{(u, v, w) \in E(F)} a_{s_u s_v s_w}(G_n) ,
 \end{align}
where $\cH_r$ is the collection all 3-uniform multi-hypergraphs with at most $r$ hyperedges and no isolated vertex, and $a_{s_u s_v s_w}(G_n) = a_{s_u s_v }(G_n)  a_{s_v s_w}(G_n) a_{s_w s_u}$. Now, let $H_{G_n}(K_3)$ be the 3-uniform hypergraph with vertex set $V(G_n)$ and a hyperedge $\bm s=(s_1, s_2, s_3)$ whenever $M(\bm s, G_n) \ne 0$, that is, there is an hyperedge between 3 vertices of  $G_n$ whenever there is a triangle passing through the vertices. 
Note that the number of hyperedges  $|E(H_{G_n}(K_3))| = N (K_3, G_n)$, which implies, by Remark \ref{rem:alon_exponent} and Lemma \ref{lm:degree_one}, 
$$\sum_{\bm s \in V(G_n)_{V(F)}} \prod_{(u, v, w) \in E(F)} a_{s_u s_v s_w}(G_n) \leq |E(H_{G_n}(K_3))|^\frac{|E(F)|}{2} \lesssim_r N (K_3, G_n)^\frac{r}{2},$$
for all $F \in \cH_r$ such that  $d_{\mathrm{min}}(F) \geq 2$. Now, the recall the definition of $d_{uv}$ from Theorem \ref{T1}. Clearly, $N(K_3, G_n)=\frac{1}{3} \sum_{1 \leq u < v \leq |V(G_n)|} d_{uv}$ and $N(\triangle_2, G_n)=\sum_{1 \leq u < v \leq |V(G_n)|} {d_{uv} \choose 2}$. This shows, recalling \eqref{eq:variance_T3},  
\begin{align}\label{eq:K3D2}
\Var(T_3(G_n)) & \asymp N(K_3, G_n) + N(\triangle_2, G_n) \nonumber \\ 
& \asymp  \sum_{1 \leq u < v \leq |V(G_n)|} d_{uv} +  \sum_{1 \leq u < v \leq |V(G_n)|} d_{uv}^2(G_n) \bm 1\{d_{uv} \geq 2\}.
\end{align} 
This implies that $\Var(T_3(G_n) \geq \sum_{1 \leq u < v \leq |V(G_n)|} d_{uv} \gtrsim N (K_3, G_n)$. Then, because the sum in \eqref{eq:T2moment} is over a finite set (not depending on $n$),
it follows that 
\begin{align}\label{eq:T2bound}
\frac{\E(T_2(K_3, G_n)^r)}{\Var(T_3(G_n))^{\frac{r}{2}}}  \lesssim_r 1.
\end{align}

Next, consider $T_1(K_3, G_n)$. Then, 
\begin{align*} 
& \frac{\E(T_1(K_3, G_n)^r)}{\Var(T_3(G_n))^{\frac{r}{2}}}  \nonumber \\
&=\frac{1}{2^r c^r \Var(T_3(G_n))^{\frac{r}{2}}} \sum_{a_1,\cdots, a_r \in [c]} \sum_{\substack{1 \leq i_1 \ne j_1 \leq |V(G_n)| \\ \vdots \\ 1 \leq i_r \ne j_r  \leq |V(G_n)|}}  \prod_{s=1}^r d_{i_s j_s}(G_n)  \E  \left(  \prod_{s=1}^r  Z_{i_s}(a_s) Z_{j_s}(a_s) \right).
\end{align*} 
Let $F$ be the unweighted multi-graph formed by the union of the edges $(i_1, j_1), (i_2, j_2), \ldots, (i_r, j_r)$. Note that 
$$  \E  \left(  \prod_{s=1}^r  Z_{i_s}(a_s) Z_{j_s}(a_s) \right) = 0, $$ 
whenever there exists a $v \in V(F)$ with $d_F(v)=1$. This implies, 
\begin{align}\label{eq:T1moment} 
 \frac{\E(T_1(K_3, G_n)^r)}{\Var(T_3(G_n))^{\frac{r}{2}}}   \lesssim_{c, r} \frac{1}{\Var(T_3(G_n))^{\frac{r}{2}}}  \sum_{F \in \cG_r: d_{\mathrm{min}}(F) \geq 2} \sum_{\bm s \in V(G_n)_{V(F)}} \prod_{(u, v) \in E(F)} d_{s_u s_v}(G_n),
\end{align}
where $\cG_r$ is the collection all multi-graphs with at most $r$ edges and no isolated vertex. 
Then by Corollary \ref{cor:alon_exponent}, for any $F \in \cG_r$ such that $d_{\mathrm{min}}(F) \geq 2$, 
$$\sum_{\bm s \in V(G_n)_{V(F)}} \prod_{(u, v) \in E(F)} d_{s_u s_v}(G_n) \lesssim_r \left(\sum_{1 \leq u \ne v \leq |V(G_n)|} d_{u v}(G_n)^2 \right)^{\frac{r}{2}}.$$ 
Then using the fact that $\Var(T_3(G_n) \gtrsim \sum_{1 \leq u < v \leq |V(G_n)|} d_{uv}^2$ (recall \eqref{eq:K3D2}) and the sum in \eqref{eq:T1moment} is over a finite set (not depending on $n$),
it follows that 
\begin{align}\label{eq:T1bound}
\frac{\E(T_1(K_3, G_n)^r)}{\Var(T_3(G_n))^{\frac{r}{2}}}  \lesssim_r 1.
\end{align}
Combining \eqref{eq:T2bound} and \eqref{eq:T1bound} the result in \eqref{eq:moment_Z3_I} follows. 

\section{Discussions and Future Directions}
\label{sec6} 

In this paper we obtain a quantitative CLT for $Z_3(G_n)$, the (standardized) number of monochromatic triangles in a uniformly random $c$-coloring of a graph sequence $G_n$. The resulting error term can be bounded in the terms of the fourth moment difference $\E(Z_3(G_n)^4)-3$ whenever $c \geq 5$.  The next natural step is to consider monochromatic $r$-cliques, for $r \geq 4$ or general monochromatic subgraphs $H$. Even though a CLT for general monochromatic subgraphs is known in a few special cases (for example, when $G_n=K_n$ is itself the complete graph, or a converging sequence of dense graphs \cite{BM_coloring_dense}), the precise conditions for it to have a Gaussian limit is yet to be understood. Given the results above, it is reasonable to expect that there will exist a critical value $c_0(H) \geq 2$ (depending on $H$) such that a fourth moment theorem for the number of monochromatic copies of $H$ in $G_n$ will hold for $c \geq c_0(H)$. While, in principle, the current approach based on the Hoeffding's decomposition and quantitative martingale CLT has the promise to generalize, the analysis of the resulting error terms become intractable as one moves from triangles to higher subgraphs. This demands a more systematic approach for understanding the various graphs that arise in the fourth moment of the corresponding statistics. It also remains to investigate the fourth-moment phenomenon in the regime where $c=c_n \rightarrow \infty$ such that $\E(Z_3(G_n))=\frac{N(K_3, G_n)}{c^2} \rightarrow \infty$. While it is possible to make the dependence on $c$ in the upper bound in  \eq{12} explicit by following the proof of Theorem \ref{T1}, 
to relate the resulting error term to the fourth-moment difference in the case $c=c_n \rightarrow \infty$,  a strengthening of Lemma~\ref{l1} (where the subgraph counts are appropriately scaled by the leading order of their corresponding moments) will be required. 

Another interesting direction is to explore whether the error bounds obtained in Theorem \ref{T1} and Theorem \ref{T2} can be improved, using other quantitative CLT methods. In this regard, in forthcoming work Omar El Dakkak, Ivan Nourdin, and Giovanni Peccati have used Stein's method to obtain a better error bound for the CLT of the number of monochromatic edges  for $c\geq 2$ fixed (personal communication). \\

\small{
\noindent\textbf{Acknowledgements}: BBB thanks Somabha Mukherjee for useful discussions during the preliminary stages of the project. XF thanks Giovanni Peccati for many helpful discussions during the preparation of the paper and Omar El Dakkak for checking the fourth moment computation in Lemma~\ref{l2}. The paper was initiated from a discussion of the first two authors during a workshop at the American Institute of Mathematics. We acknowledge the financial support of the Institute and thank the staff for their hospitality. XF was partially supported by Hong Kong RGC ECS 24301617 and GRF 14302418, and 14304917, a CUHK direct grant, and a CUHK start-up grant.}


\appendix

\section{Estimates from Extremal Combinatorics}
\label{sec:moment_hypergraph}

In this section we collect some estimates from extremal combinatorics, which might be useful in estimating graph counts. 

\begin{definition} Given a $s$-uniform multi-hypergraph $F=(V(F), E(F))$, the {\it fractional stable number} of $F$, to be denoted by $\gamma(F)$, is defined as: 
\begin{equation}
\gamma(F)= \arg\max_{\substack{\phi: V(F) \rightarrow [0,1], \\ \sum_{x \in e}\phi(x) \le 1 \text{ for every }  e \in E(F)}} \sum_{v\in V(F)} \phi(v). 
\label{eq:gamma}
\end{equation}
It is clear  that $\gamma(F)=\gamma(F_\mathrm{sim})$, where $F_{\mathrm{sim}}$ is the simple hypergraph obtained from $F$ by replacing the hyperedges between the vertices which occur more than once, by a single hyperedge.  
\end{definition}

Recall that $d_{\min}(F)$ denotes the minimum degree of a hypergraph $F$. The next lemma is a generalization of \cite[Lemma 4.1]{BDM} for hypergraphs. 

\begin{lemma}
Let $F=(V(F), E(F))$ be a multi-hypergraph with no isolated vertex and $d_{\min}(F)\geq 2$. Then $\gamma(F)\leq \frac{1}{2}|E(F)|$. 
\label{lm:degree_one}
\end{lemma}

\begin{proof}  Let $\varphi:V(F)\rightarrow [0, 1]$ be an optimal solution of the linear program (\ref{eq:gamma}). If $d_{\min}(F)\geq 2$, we have
$$ \sum_{x\in V(F)} \varphi(x) \leq \frac{1}{d_{\min}(F)} \sum_{x\in V(F)} d_F(x) \varphi(x)\leq \frac{1}{d_{\min}(F)} \sum_{e \in E(F)} \sum_{x \in e} \varphi(x) \leq\frac{1}{2}|E(F)|.$$
which gives $\gamma(F)\leq \frac{1}{2}|E(F)|$.
\end{proof}

A weighted $r$-uniform hypergraph $H=(V(H), w)$ is a function $w: V(H)_r \rightarrow \R_{\geq 0}$, which is symmetric under the permutation of the coordinates. If the function $w$ takes values in $\{0, 1\}$, then the collection of $r$-tuples where $w=1$ is the collection of hyperedges $E(H)$, which gives the usual unweighted hypergraph. The following lemma gives a bound on the number of weighted copies of a hypergraph $F$ in a weighted hypergraph $H$. 

\begin{theorem}{\em(\cite[Lemma~3.3]{hypergraphs})} For a fixed (unweighted) multi-hypergraph $F=(V(F), E(F))$ and a weighted $r$-uniform hypergraph $H=(V(H), w)$, there exists a positive constant  $C=C(F)$, such that \begin{align}
\sum_{\bm s \in V(H)_{V(F)}} & \prod_{(u_1, u_2, \ldots, u_r) \in E(F)} w( s_{u_1}, \bm s_{u_2}, \ldots,  s_{u_r}) \nonumber \\ 
& \leq C  \prod_{ (u_1, u_2, \ldots, u_r) \in E(F)} \left(\sum_{\bm e \in V(H)_r} w(\bm e)^{\frac{1}{\phi(u_1, u_2, \ldots, u_r)}} \right)^{\phi(u_1, u_2, \ldots, u_r)},
\label{eq:alon_exponent}
\end{align}
for any function $\phi: E(F) \rightarrow [0, 1]$ such that 
$$\sum_{\bm u \in E(F): \bm u \text{ containing } v} \phi(\bm u) \geq 1, \quad \text{ for all } v \in V(F).$$  
\label{th:alon_exponent}
\end{theorem}

\begin{remark}\label{rem:alon_exponent} The result above was proved for unweighted graphs by \cite{alon81} and for unweighted hypergraphs by \cite{hypergraphcopies}. In this case, $w(\bm e)^{\frac{1}{\phi(u_1, u_2, \ldots, u_r)}}=w(\bm e)$, and $ \sum_{\bm e \in V(H)_s} w(\bm e) =|E(H)| $, the number of hyperedges in $H$. Therefore, the bound above becomes $N(F, H) \leq C |E(H)|^{\alpha(F)} = C |E(H)|^{\gamma(F)}$, 
where 
$$\alpha(F):= \min_{\phi: E(F) \rightarrow [0, 1]} \sum_{\bm u \in E(F)} \phi (\bm u) \quad \text{ such that }\sum_{\bm u \text{ containing } v} \phi(\bm u) \geq 1, \quad \text{ for all } v \in V(F).$$
The quantity $\alpha(F)$ is called the fractional edge-cover number of $F$, and $\gamma(F)=\alpha(F)$ by the linear programming duality. 
\end{remark}

If the hypergraph $F$ has minimum degree 2, then we can choose $\phi(\bm u)=\frac{1}{2}$, for all $\bm u \in E(F)$ in the previous result. This gives the following bound. 

\begin{corollary} For a fixed (unweighted) hypergraph $F=(V(F), E(F))$ with $d_{\mathrm{min}}(F) \geq 2$, and a weighted $r$-uniform hypergraph $H=(V(H), w)$, there exists a positive constant  $C=C(F)$, such that 
\begin{align}
\sum_{\bm s \in V(H)_{V(F)}} & \prod_{ (u_1, u_2, \ldots, u_r) \in E(F)} w( s_{u_1},  s_{u_2}, \ldots,  s_{u_r}) \leq C   \left(\sum_{\bm e \in V(H)_r} w(\bm e)^2 \right)^{\frac{|E(F)|}{2}}.
\label{eq:alon_exponent_F}
\end{align}
\label{cor:alon_exponent}
\end{corollary}

\section{List of Subgraphs and their Coefficients in the Fourth Moment Difference}\label{app2}

In this section we list the set of graphs which contribute to the fourth moment difference $\E(Z_3(G_n)^4)-3$ (as in Lemma \ref{l2}) in Figure \ref{fig2} and their corresponding coefficients in Table \ref{table:coefficient_list}.

\begin{figure}[h]
\centering
\begin{minipage}[l]{1.0\textwidth}
\centering
\includegraphics[width=4.15in]
    {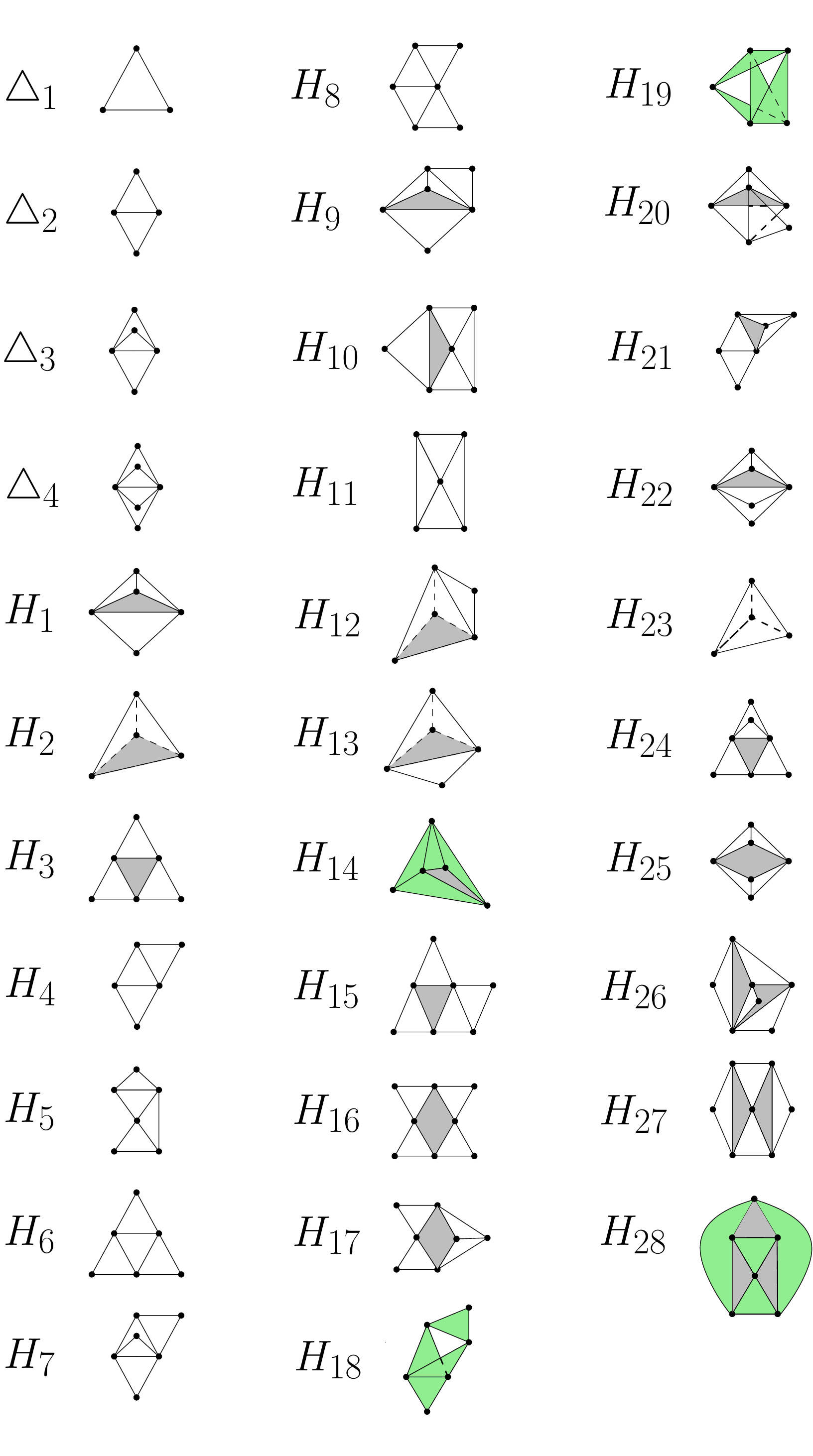}\\
\end{minipage}
\caption{\small{The different subgraphs that arise in the fourth moment condition $\E (Z_3(G_n)^4)-3$. 
Each subgraph has several specified triangles which are shown in green for $H_{14}$, $H_{18}$, $H_{19}$ and $H_{28}$ and in white for other subgraphs (1 triangle for $\triangle_1$, 2  triangles for $\triangle_2$, 3 triangles for $\triangle_3$, 4 triangles for $\triangle_4$, 3 triangles for $H_1$--$H_4$, and 4 triangles for $H_5$--$H_{28}$). Grey areas are used to contrast white. Dashed lines are used for 3-dimensional pictures. In Section~\ref{sec:th4thmomentpf_I}, $N(H, G_n)$ counts the number of copies of $H$ with the specified triangles in $G_n$. For example, a tetrahedron in $G_n$ is counted four times in $N(H_2, G_n)$ (once each to account for the missing grey face), and one time in $N(H_{23}, G_n)$. 
}}
\label{fig2}
\end{figure}

\newpage 

\begin{table} 
\small 
\centering
\begin{tabular}{|c|l|} 
\hline
Coefficient & \quad \quad Value of the Coefficient  \\
\hline
$\delta_1$ & \quad \quad  $\frac{1}{c^2}- \frac{7}{c^4}+\frac{12}{c^6}-\frac{6}{c^8}$ \\
$\delta_2$ & \quad \quad   $\frac{14}{c^3}-\frac{14}{c^4}-\frac{72}{c^5}+\frac{60}{c^6}+\frac{96}{c^7}-\frac{84}{c^8}$  \\
$\delta_3$ & \quad \quad  $36\left(\frac{1}{c^4}-\frac{3}{c^5}-\frac{2}{c^6}+\frac{10}{c^7}-\frac{6}{c^8} \right)$ \\
$\delta_4$ & \quad \quad  $24\left(\frac{1}{c^5}-\frac{7}{c^6}+\frac{12}{c^7}-\frac{6}{c^8}\right)$ \\
$h_1$ & \quad \quad  $36\left(\frac{1}{c^4}-\frac{1}{c^5}-\frac{2}{c^6}+\frac{2}{c^7}\right)$ \\
$h_2$ & \quad \quad  $36\left(\frac{1}{c^3}-\frac{5}{c^5}+\frac{10}{c^7}-\frac{6}{c^8}\right)$ \\
$h_3$ & \quad \quad  $36\left(\frac{1}{c^5}-\frac{1}{c^6}-\frac{2}{c^7}+\frac{2}{c^8}\right)$ \\
$h_4$ & \quad \quad  $12\left(\frac{3}{c^4}-\frac{6}{c^5}-\frac{5}{c^6}+\frac{16}{c^7}-\frac{8}{c^8}\right)$  \\
$h_5$ & \quad \quad  $24\left(\frac{1}{c^5}-\frac{3}{c^6}+\frac{3}{c^7}-\frac{1}{c^8}\right)$ \\
$h_6$ & \quad \quad  $24\left(\frac{1}{c^5}-\frac{3}{c^6}+\frac{2}{c^7}\right)$ \\
$h_7$ & \quad \quad  $24\left(\frac{1}{c^5}-\frac{4}{c^6}+\frac{5}{c^7}-\frac{2}{c^8}\right)$ \\
$h_{8}$ & \quad \quad  $24\left(\frac{1}{c^5}-\frac{3}{c^6}+\frac{3}{c^7}-\frac{1}{c^8}\right)$ \\
$h_{9}$ & \quad \quad  $24\left(\frac{1}{c^5}-\frac{2}{c^6}+\frac{1}{c^8}\right)$ \\
$h_{10}$ & \quad \quad  $24\left(\frac{1}{c^5}-\frac{1}{c^6}-\frac{1}{c^7}+\frac{1}{c^8}\right)$ \\
$h_{11}$ & \quad \quad  $24\left(\frac{1}{c^4}-\frac{6}{c^6}+\frac{8}{c^7}-\frac{3}{c^8}\right)$ \\
$h_{12}$ & \quad \quad  $24\left(\frac{1}{c^4}-\frac{1}{c^5}-\frac{5}{c^6}+\frac{9}{c^7}-\frac{4}{c^8}\right)$ \\
$h_{13}$ & \quad \quad  $24\left(\frac{1}{c^4}-\frac{1}{c^5}-\frac{4}{c^6}+\frac{6}{c^7}-\frac{2}{c^8}\right)$ \\
$h_{14}$ & \quad \quad  $24\left(\frac{1}{c^4}-\frac{5}{c^6}+\frac{5}{c^7}-\frac{1}{c^8}\right)$ \\
$h_{15}$ & \quad \quad  $24\left(\frac{1}{c^6}-\frac{2}{c^7}+\frac{1}{c^8}\right)$ \\
$h_{16}$ & \quad \quad  $24\left(\frac{1}{c^7}-\frac{1}{c^8}\right)$ \\
$h_{17}$ & \quad \quad  $24\left(\frac{1}{c^6}-\frac{1}{c^7}\right)$ \\
$h_{18}$ & \quad \quad  $24\left(\frac{1}{c^5}-\frac{2}{c^6}+\frac{1}{c^7}\right)$ \\
$h_{19}$ & \quad \quad  $24\left(\frac{1}{c^4}-\frac{4}{c^6}+\frac{3}{c^7}\right)$ \\
$h_{20}$ & \quad \quad  $24\left(\frac{1}{c^5}-\frac{1}{c^6}-\frac{2}{c^7}+\frac{2}{c^8}\right)$ \\
$h_{21}$ & \quad \quad  $24\left(\frac{1}{c^5}-\frac{2}{c^6}+\frac{1}{c^7}\right)$ \\
$h_{22}$ & \quad \quad  $24\left(\frac{1}{c^5}-\frac{3}{c^6}+\frac{2}{c^7}\right)$ \\
$h_{23}$ & \quad \quad  $24\left(\frac{1}{c^3}- \frac{4}{c^5}- \frac{3}{c^6}+\frac{12}{c^7}- \frac{6}{c^8}\right)$ \\
$h_{24}$ & \quad \quad  $24\left(\frac{1}{c^6}- \frac{3}{c^7}+ \frac{2}{c^8}\right)$ \\
$h_{25}$ & \quad \quad  $24\left(\frac{1}{c^5}-\frac{1}{c^6}\right)$ \\
$h_{26}$ & \quad \quad  $24\left(\frac{1}{c^6}- \frac{3}{c^7}+ \frac{2}{c^8}\right)$ \\
$h_{27}$ & \quad \quad  $24\left(\frac{1}{c^6}- \frac{2}{c^7}+\frac{1}{c^8}\right)$ \\
$h_{28}$ & \quad \quad  $24\left(\frac{1}{c^5}-\frac{4}{c^7}+ \frac{3}{c^8} \right)$\\
\hline
\end{tabular}
\caption{\small{The values of the coefficients for the different subgraphs that contribute to $\E(Z_3(G_n)^4)-3$.}} 
\label{table:coefficient_list} 
\end{table}

\normalsize

\end{document}